\documentclass[a4paper, 12pt]{article}
\usepackage[a4paper, left=2.5cm, right=2.5cm, top=2.5cm, bottom=2cm]{geometry}
\usepackage[utf8]{inputenc}
\usepackage{lmodern} 
\usepackage[english]{babel}
\usepackage{ngerman}
\usepackage{amsmath}
\numberwithin{equation}{section}
\usepackage{amsfonts}
\usepackage{amsthm}
\usepackage{amssymb}
\usepackage[ruled,vlined]{algorithm2e}
\usepackage{graphicx}
\usepackage{hyperref}

\newtheorem{theorem}{Theorem}[section] % reset theorem numbering for each chapter

\newtheorem{definition}[theorem]{Definition}
\newtheorem{remark}[theorem]{Remark}
\newtheorem{assumption}[theorem]{Assumption}
\newtheorem{lemma}[theorem]{Lemma}

\newtheorem{proposition}[theorem]{Proposition}
\newtheorem{example}[theorem]{Example}

\usepackage{accents}

\title{Error analysis for probabilities of rare events with approximate models}
\author{F. Wagner, J. Latz, I. Papaioannou, E. Ullmann}
\date{\today}

\begin{document}

\maketitle

\begin{abstract}
The estimation of the probability of rare events is an important task in reliability and risk assessment. We consider failure events that are expressed in terms of a limit-state function, which depends on the solution of a partial differential equation (PDE). In many applications, the PDE cannot be solved analytically. We can only evaluate an approximation of the exact PDE solution. Therefore, the probability of rare events is estimated with respect to an approximation of the limit-state function. This leads to an approximation error in the estimate of the probability of rare events. 
Indeed, we prove an error bound for the approximation error of the probability of failure, which behaves like the discretization accuracy of the PDE multiplied by an approximation of the probability of failure, the first order reliability method (FORM) estimate. This bound requires convexity of the failure domain. For non-convex failure domains, we prove an error bound for the relative error of the FORM estimate. Hence, we derive a relationship between the required accuracy of the probability of rare events estimate and the PDE discretization level. This relationship can be used to guide practicable reliability analyses and, for instance, multilevel methods.
\vspace{0.5cm}
\end{abstract}

\section{Introduction}\label{Sec: Introduction}
The distinction of safe and failure events is a crucial topic in reliability analysis and risk management. The occurrence of failure events cannot always be avoided; therefore, the estimation of the probability of such occurrences is of high significance. Indeed, failure probabilities are usually small; hence, the denomination \emph{probability of rare events} is commonly used. The failure event can be defined in terms of a \emph{limit-state function} (LSF). The LSF is a function of a set of uncertain parameters and the failure event is defined by the collection of parameter values for which the LSF takes non-positive values.
\\In this work, we consider settings where the LSF is based on the solution of an elliptic partial differential equation (PDE) with random diffusion coefficient. These situations frequently arise in engineering risk settings. For example, the authors in \cite{Cornaton08,Noseck08} consider radioactive waste repositories. Therein, the departure of radioactive particles and their travel paths through the subsurface are of high relevance. The goal is to determine the probability that radioactive particles come back to the human environment. Since the exact subsurface properties and exact travel paths of the particles are unknown, the hydraulic conductivity of the soil is modelled as a random field while the particle flow is simulated by a finite element method (FEM) approximation of the groundwater flow and transport equation.
\\The application of discretization schemes, such as Finite Differences \cite{LeVeque07}, Finite Volumes \cite{Eymard00} or FEM \cite{Braess07}, introduce a PDE discretization error in the evaluation of the LSF. Consequently, this leads to an approximation error of the probability of rare events. The accuracy of the approximation depends on the discretization size of the spatial or temporal domain. In previous years, many methods have been developed for rare event estimation, which are based on a sequence of discretization levels with increasing accuracy. Examples are Multilevel Subset Simulation \cite{Ullmann15}, Multilevel Monte Carlo \cite{Elfverson16}, Multilevel Sequential Importance Sampling \cite{Wagner20} or a multifidelity approach in \cite{Peherstorfer18}. However, only the sampling error of the methods has been considered so far. In contrast, we do not consider the sampling error and focus on the PDE discretization error.
\\There is a large amount of literature available, which derive error bounds for the PDE discretization error. However, there are few publications which consider the induced approximation error of rare event probabilities. The authors in \cite{Elfverson16} derive an upper bound for the absolute approximation error of the probability of failure which behaves as the PDE discretization error. However, the absolute error is of limited interest since failure probabilities are usually small. This manuscript closes this gap for LSFs which are based on elliptic PDEs. Indeed, the work of \cite{Elfverson16} forms the starting point for our contributions.
\\We derive an error bound for the probability of rare events which behaves as the PDE discretization error multiplied by the \emph{first order reliability method} (FORM) \cite{Hasofer74,Kiureghian04} estimate. The FORM estimate is determined by the minimum distance of the failure domain to the origin of an independent standard Gaussian input space. If the failure domain is a convex set, the FORM estimate is an upper bound for the probability of the rare event. We use this condition as an assumption for the derived error bound. Indeed, if the FORM estimate is equal to the probability of failure, our error bound gives an upper bound for the relative error. An example for this case is an LSF which is affine linear with respect to Gaussian stochastic parameters. Moreover, we provide an error bound of the relative error with respect to the FORM estimates. This bound is more generally applicable since convexity of the failure domains is not required.
\\The manuscript is structured as follows. In Section~\ref{Sec Setting}, we present the main theorem, underlying setting and relevant assumptions. The proof of the theorem is given in the subsequent sections. First, we show that the absolute error is bounded by the PDE discretization error multiplied by the local Lipschitz constant associated with the \emph{cumulative distribution function} (CDF) of the LSF. In Section~\ref{Sec linear lsf}, we show that for affine linear LSFs, the upper bound of the local Lipschitz constant depends linearly on the probability of failure. This gives an upper bound for the relative error. If the LSF is based on the solution of an elliptic PDE with stochastic diffusion coefficient, we show in Section~\ref{Sec. Opt Control} that the distance of the exact and approximate failure domains behaves as the PDE discretization error. Thereafter, we show that the Gaussian measure of the symmetric difference of the failure domains can be bounded by the Gaussian measure of an interval in 1D if the failure domains are convex. Using this fact, the LSF is linearized around the \emph{most likely failure point} (MLFP) and the results for affine linear LSFs are applied. This proves the main theorem. In Section~\ref{Sec. experiments}, we consider three numerical examples, two in low dimensions with analytical solutions, and one high-dimensional example.

\section{Problem Setting and Main Result}\label{Sec Setting}
Even though many of our results are applicable to more general settings, we focus in this work on failure events that are based on the solution of an elliptic PDE with random diffusion coefficient and Dirichlet boundary conditions. In the following subsections, we first introduce this setting more particularly. Then, we briefly describe the FORM approach to the estimation of rare event probabilities. The FORM estimate is part of the error bound that forms the main result of this work and that we summarise in Section \ref{subsection error bound}.

\subsection{Elliptic PDE, failure events, and their approximation}
Let $(\Omega, \mathcal{A}, \mathbb{P})$ be a probability space and $D\subset \mathbb{R}^d$, $d=1,2,3$, be an open, bounded, convex, polygonal domain. The given quantities are a real-valued random field $a:D\times \Omega\rightarrow \mathbb{R}$ and a real valued function $f\in L^2(D)$. We seek a random field $y:\overline{D}\times\Omega\rightarrow\mathbb{R}$, such that for $\mathbb{P}$-almost every (a.e.) $\omega\in\Omega$ it holds
\begin{align}
-\nabla_x\cdot(a(x,\omega)\nabla_x y(x,\omega)) &= f(x) \quad \forall x\in D,\label{elliptic PDE}
\\y(x,\omega) &= 0 \quad \forall x\in\partial D.\label{boundary term}
\end{align}
In practice, we employ a FEM discretization to solve~\eqref{elliptic PDE}.
Thus, we consider the weak or variational form of the PDE. A random field $y:\overline{D}\times\Omega\rightarrow\mathbb{R}$ satisfies the \emph{pathwise variational formulation}, if for a fixed $\omega\in\Omega$ it holds that $y(\cdot,\omega)\in V$ and
\begin{align}
\int_D a(x,\omega)\nabla_x y(x,\omega) \cdot \nabla_x v(x) \mathrm{d}x = \int_D f(x) v(x) \mathrm{d}x \quad\forall v\in V,\label{pathwise}
\end{align}
where $V:=H_0^1(D)$. Let $h>0$ denote a discretization parameter, typically the mesh size. We define the \emph{discretized pathwise variational formulation} for $y_h\in V_h$ as
\begin{align}
\int_D a(x,\omega)\nabla_x y_h(x,\omega) \cdot \nabla_x v_h(x) \mathrm{d}x = \int_D f(x) v_h(x) \mathrm{d}x \quad\forall v_h\in V_h.\label{pathwise discrete}
\end{align} 
Here, $V_h\subset V$ is a finite dimensional vector space. In this manuscript, we consider two types of diffusion coefficients.
\begin{definition}[Ellipticity and boundedness of the diffusion coefficient]\label{Defintion diffusion coeff}
\hspace{0.001cm}
\begin{itemize}
\item[(I)] The diffusion coefficient $a$ is uniformly elliptic and bounded if there exists $a_{\mathrm{min}}, a_{\mathrm{max}}> 0$ such that for $\mathbb{P}$-a.e. $\omega\in\Omega$
\begin{align}
0< a_{\mathrm{min}}\le a(x,\omega)\le a_{\mathrm{max}} <\infty, \quad \text{for a.e. } x\in D.\label{Def uniform}
\end{align}
\item[(II)] The diffusion coefficient $a$ is pathwise elliptic and bounded if there exists real-valued random variables $a_{\mathrm{min}}, a_{\mathrm{max}}: \Omega\rightarrow\mathbb{R}$ such that for $\mathbb{P}$-a.e. $\omega\in\Omega$
\begin{align}
0< a_{\mathrm{min}}(\omega)\le a(x,\omega)\le a_{\mathrm{max}}(\omega) <\infty, \quad \text{for a.e. } x\in D,\label{Def pathwise}
\end{align}
where $a_{\mathrm{min}}(\omega):=\underset{x\in\overline{D}}{\min}\hspace{0.2cm} a(x,\omega)$ and $a_{\mathrm{max}}(\omega):=\underset{x\in\overline{D}}{\max}\hspace{0.2cm} a(x,\omega)$.
\end{itemize}
\end{definition}

\begin{remark}\label{Remark non-uniformly elliptic}
We note that our proved error bounds in Proposition~\ref{proposition} and Theorem~\ref{Thm error bound} require that the approximation error of an observation operator acting on $y$ and $y_h$ is uniformly bounded; see Assumption~\ref{Ass LSF error}. This assumption is in general violated for diffusion coefficients which are only pathwise elliptic and bounded, i.e., they do not satisfy~\eqref{Def uniform}. Therefore, Proposition~\ref{proposition} and Theorem~\ref{Thm error bound} are only valid for uniformly elliptic and bounded diffusion coefficients. In Remark~\ref{Remark uniform} and~\ref{Remark error bound}, we will discuss in which way our derived error bounds are useful for diffusion coefficients which only satisfy~\eqref{Def pathwise}.

\end{remark}
Under Assumption~\ref{regularity of a}, existence and uniqueness of a solution for~\eqref{pathwise} and~\eqref{pathwise discrete} is ensured by \cite[Theorem 9.9]{Lord14}. Moreover, under Assumption~\ref{regularity of a} and for $d=2$, the authors of \cite[Theorem 2.1]{Teckentrup2013} show that the solution $y$ of~\eqref{pathwise} satisfies $y(\cdot,\omega)\in H^2(D)\cap H_0^1(D)$, which we require in the proof of Theorem~\ref{Thm. LSF distance}. For $d=3$, the authors state in~\cite[Remark 5.2 (c)]{Teckentrup2013} that the same property holds if $D$ is convex.

\begin{assumption}[Regularity of the diffusion coefficient]\label{regularity of a}
We assume that
\begin{itemize}
\item[(i)] the computational domain $D$ is open, bounded, convex and polygonal,
\item[(ii)] $a_{\mathrm{min}}(\omega)\ge 0$ for $\mathbb{P}$-a.e. $\omega\in\Omega$ and $1/a_{\mathrm{min}}\in L^p(\Omega)$ for all $p\in(0,\infty)$,
\item[(iii)] $a\in L^p(\Omega, C^1(\bar{D}))$ for all $p\in(0,\infty)$, i.e., the realisations $a(\cdot, \omega)$ are continuously differentiable,
\item[(iv)] $f\in L^2(D)$.
\end{itemize}
\end{assumption}
We note that Assumption~\ref{regularity of a} (ii) is automatically satisfied for uniformly elliptic and bounded diffusion coefficients. Moreover, $y(\cdot,\omega)\in H^2(D)\cap H_0^1(D)$ is still ensured if Assumption~\ref{regularity of a} (i) is replaced by requiring that $D$ is open, bounded and has a $C^2$ boundary \cite[Theorem 3.4]{Scheichl13}.

Having considered the spatial regularity of the diffusion coefficient, we specify the parametric regularity of $a$ in the following assumption. Moreover, we require that $a$ depends on a Gaussian random variable.

\begin{assumption}[Parametric form and parametric regularity of the diffusion coefficient]\label{Ass gaussian}
\hspace{0.001cm}
\begin{itemize}
\item[(i)] The diffusion coefficient $a(x,\omega)$ is a measurable function of an $n$-variate random vector $U:\Omega\rightarrow\mathbb{R}^n$, where $U$ follows the $n$-variate independent standard normal distribution. This means, there is a function $\widehat{a}: D\times\mathbb{R}^n\rightarrow\mathbb{R}$ with $a(x,\omega) = \widehat{a}(x,U(\omega))$ for $\mathbb{P}$-a.e. $\omega\in\Omega$.
\item[(ii)] The diffusion coefficient $a(x,\omega)$ is three times continuously differentiable with respect to outcomes $u\in\mathbb{R}^n$ of $U$ for all $x\in D$.
\end{itemize} 
\end{assumption}
Assumption~\ref{Ass gaussian}~(i) implies that $a(x,\omega)$ can be viewed as a function in space depending on an $n$-dimensional parameter given by the outcomes $u\in\mathbb{R}^n$ of $U$. Thus, $a$ can be viewed as finite dimensional noise \cite[Definition 9.38]{Lord14}. We note that Assumption~\ref{Ass gaussian}~(i) is not a strong restriction. Under mild assumptions, a non-Gaussian random variable $\widetilde{U}$ can be transformed via an isoprobabilistic transformation $U=T(\widetilde{U})$ to a Gaussian random variable $U$. For instance, if $\widetilde{U}$ can be modelled by a Gaussian copula, the Nataf transform \cite{Kiureghian86} can be applied to express it as a function of a standard normal random variable. If the conditional distributions of $\widetilde{U}_{k+1}$ given $\widetilde{U}_1,\dots,\widetilde{U}_{k}$ are known for $k=1,\dots,n-1$, the Rosenblatt transform can be applied \cite{Hohenbichler81}. 

Based on the elliptic PDE, we now define the LSF, the failure event, the failure probability, and their approximations.
Failure is defined in terms of an LSF $G:\mathbb{R}^n\rightarrow\mathbb{R}$ such that $G(U(\omega))\le 0$ for $\omega\in\Omega$. Furthermore, we assume that the LSF $G$ and the PDE solution $y$ are related via a linear and bounded operator $\mathcal{F}:V\rightarrow\mathbb{R}$
\begin{align}
G(U(\omega)) := y_{\mathrm{max}} - \mathcal{F} y(\cdot, \omega), \label{LSF form}
\end{align} 
where $y_{\mathrm{max}}\in\mathbb{R}$ is a constant. Analogously, we define the discretized LSF $G_h:\mathbb{R}^n\rightarrow\mathbb{R}$ as
\begin{align}
G_h(U(\omega)) := y_{\mathrm{max}} - \mathcal{F}_h y_h(\cdot,\omega),\label{LSF form discrete}
\end{align} 
where $\mathcal{F}_h:V_h\rightarrow\mathbb{R}$ is the induced discretization of $\mathcal{F}$. With the operator $\mathcal{F}$, we define the dual problem, where we seek the solution $z(\cdot,\omega)\in H_0^1(D)$ such that
\begin{align}
\int_D a(x,\omega)\nabla_x z(x,\omega)\cdot\nabla_x v(x)\mathrm{d}x = \mathcal{F}(v) \quad \forall v\in H_0^1(D).\label{dual problem}
\end{align} 
Since $\mathcal{F}$ is linear and bounded, existence and uniqueness of a solution of the dual problem~\eqref{dual problem} is ensured by the Lax--Milgram theorem~\cite[Section 6.2.1]{evans10}. By Assumption~\ref{regularity of a}, it follows that $z(\cdot,\omega)\in H^2(D)\cap H_0^1(D)$, which we require in the proof of Theorem~\ref{Thm. LSF distance}.

Our analysis is performed for the \emph{probability of failure}. This quantity is defined as the probability mass of the \emph{failure domain} $A:= \{u\in\mathbb{R}^n: G(u)\le 0\}$, which is expressed as
\begin{align}
P_{f} := \mathbb{P}[A] = \mathbb{P}[G(U)\le0] = \int_{u\in\mathbb{R}^n} I(G(u)\le0) \varphi_n(u) \mathrm{d}u,\label{probability of failure}
\end{align}
where $I$ denotes the indicator function; i.e., $I(\mathrm{true}) = 1$ and $I(\mathrm{false}) = 0$. The function $\varphi_n: \mathbb{R}^n \rightarrow \mathbb{R}$ denotes the \emph{probability density function} (PDF) of the $n$-variate independent standard normal distribution, which we denote by $\mathrm{N}(0, \mathrm{Id}_n)$. Replacing $G$ by $G_{h}$ in~\eqref{probability of failure} gives the approximation 
\begin{align}
P_{f,h} := \mathbb{P}[A_h] = \mathbb{P}[G_h(U)\le 0] = \int_{u\in\mathbb{R}^n} I(G_h(u)\le0)\varphi_n(u)\mathrm{d}u,\label{approx probability of failure}
\end{align}
where $A_h=\{u\in\mathbb{R}^n: G_h(u)\le 0\}$. Since $P_{f,h}$ includes numerical errors due to approximating the exact LSF $G$, we cannot expect equality of $P_f$ and $P_{f,h}$. The main contribution of this work is the derivation of an upper bound for the error 
\begin{equation} \label{Err_proboffailure}
\vert P_f - P_{f,h}\vert\le\mathbb{P}[A \triangle A_h],
\end{equation}
where $$A \triangle A_h := (A\setminus A_h)\cup (A_h\setminus A) = \{u\in\mathbb{R}^n : \text{ either } u \in A \text{ or } u \in A_h\}$$ is the \emph{symmetric difference} of $A$ and $A_h$.
This upper bound behaves as a certain approximation to the rare event probability with the approximate model multiplied with the discretization error of $\vert \mathcal{F}y - \mathcal{F}_h y_h\vert$.
For the latter, we assume the following approximation property.

\begin{assumption}[Approximation error of the LSF]\label{Ass LSF error}
The operator $\mathcal{F}$ is linear and bounded and there exists constants $C_{\mathrm{FE}}>0$ and $s>0$ independent of $h$, such that the discretization error with respect to the solution of~\eqref{pathwise} and~\eqref{pathwise discrete} satisfies for $\mathbb{P}\text{-a.e. }\omega\in\Omega$
\begin{align}
\vert G(U(\omega)) - G_h(U(\omega))\vert = \vert \mathcal{F}y(\cdot, \omega) - \mathcal{F}_h y_h(\cdot,\omega)\vert \le C_{\mathrm{FE}}h^s.\label{ass uniform bound}
\end{align}
\end{assumption}
Moreover, we require Lipschitz continuity of the CDFs of the random variables $G(U)$ and $G_h(U)$.
\begin{assumption}[Regularity of the CDFs of $G(U)$, $G_h(U)$]\label{Ass lipschitz}
The CDFs of the random variables $G(U)$ and $G_h(U)$ are local Lipschitz continuous with Lipschitz constants $C_L>0$ and $C_{L,h}>0$, i.e., for $a,b$ with $a<b$ it holds
\begin{align*}
\mathbb{P}[G(U)\in ]a,b]] &= \mathbb{P}[G(U)\le b] - \mathbb{P}[G(U)\le a] \le C_L \vert a - b \vert,
\\\mathbb{P}[G_h(U)\in ]a,b]] &= \mathbb{P}[G_h(U)\le b] - \mathbb{P}[G_h(U)\le a] \le C_{L,h} \vert a - b \vert.
\end{align*}
\end{assumption}

\begin{remark}\label{Remark uniform}
Note that the uniform bound in~\eqref{ass uniform bound} might not be valid for diffusion coefficients which are only pathwise elliptic and bounded, i.e., are of type (II) in Definition~\ref{Defintion diffusion coeff}. In this case, $C_{\mathrm{FE}}$ is a random variable and depends on $\omega\in\Omega$ such that
\begin{align*}
\vert G(U(\omega)) - G_h(U(\omega))\vert = \vert \mathcal{F}y(\cdot, \omega) - \mathcal{F}_h y_h(\cdot,\omega)\vert \le C_{\mathrm{FE}}(\omega)h^s.
\end{align*}
One idea to handle such cases is the restriction of the random variable $U$ to a bounded domain $B_R=\{u\in\mathbb{R}^n: \Vert u \Vert_2\le R\}$, where $R>0$. Thus, the random variable $C_{\mathrm{FE}}(\omega)$ is bounded uniformly within $\{\omega\in\Omega: U(\omega)\in B_R\}$. This idea can be seen as truncating the tails of the $n$-variate normal distribution. For log-normal random fields, the truncation yields a uniformly elliptic and bounded diffusion coefficient which satisfies~\eqref{ass uniform bound}. In Remark~\ref{Remark error bound}, we further discuss this idea and investigate our provided error bounds for pathwise elliptic and bounded diffusion coefficients.
\end{remark}
In fact, the following example considers a case, where $C_{\mathrm{FE}}$ depends on $\omega\in\Omega$. The diffusion coefficient $a_{\mathrm{II}}$ is a log-normal random field and, thus, not uniformly elliptic and bounded.

\begin{example}\label{example 1}
We consider the model problem~\eqref{elliptic PDE} and~\eqref{boundary term} with $D=(0,1)$, $f(x)=0$ for all $x\in D$ and boundary conditions $y(0,\omega)=1$ and $y(1,\omega)=0$. A similar problem is considered in \cite{Straub16}. We consider two examples for the diffusion coefficient
\begin{align*}
a_{\mathrm{I}}(x,\omega) := 2 + \tanh(Z(x,\omega)),\quad\quad a_{\mathrm{II}}(x,\omega) := \exp(Z(x,\omega)).
\end{align*}
In both examples, $Z$ is a Gaussian random field. We note that $a_{\mathrm{I}}$ is uniformly elliptic and bounded, i.e.~\eqref{Def uniform} is satisfied, since $1\le a_{\mathrm{I}}(x,\omega)\le 3$ for $\mathbb{P}$-a.e. $\omega\in\Omega$ and all $x\in D$. The authors of \cite{Fenton03} employ a similar random field model to describe a geotechnical material parameter. The diffusion coefficient $a_{\mathrm{II}}$ is only pathwise elliptic and bounded, i.e.~\eqref{Def pathwise} is satisfied. We assume that $Z$ has constant mean $\mu_Z$ and constant variance $\sigma_Z^2$, while the covariance function is of exponential type. It is well known that the exponential covariance kernel $c(x,y):=\exp(\Vert x- y\Vert_1/\lambda)$ produces realisations which are not continuously differentiable \cite[Chapter 4]{Rasmussen06}. The parameter $\lambda$ denotes the correlation length. The random field $Z$ can be approximated via its truncated Karhunen--Lo\`{e}ve expansion (KLE)
\begin{align*}
Z(x,\omega) \approx Z_n(x,\omega):=\mu_Z + \sigma_Z\sum_{m=1}^{n}\sqrt{\nu_m}z_m(x)U_m(\omega),
\end{align*}
where $(\nu_m, z_m)$ are the KL eigenpairs of the correlation operator. For the exponential covariance kernel, a derivation of the eigenpairs is given in \cite[Section 2.3.3]{Ghanem91}. We note that the approximation error introduced by the truncation of the KLE is not part of our analysis and the truncation order $n$ is fixed.
\\ The random variables $\{U_m\}_{m=1}^n$ are independent and standard normally distributed. Since we consider finitely many KL terms and the eigenfunctions $z_m(\cdot)$ are smooth, the realisations $a_{\mathrm{I},n}(\cdot,\omega):=2+\tanh(Z_n(\cdot,\omega))$ and $a_{\mathrm{II},n}(\cdot,\omega):=\exp(Z_n(\cdot,\omega))$ are sufficiently smooth and Assumptions~\ref{regularity of a} and~\ref{Ass gaussian} are satisfied. By the Sobolev embedding theorem \cite[Theorem 6.48]{Hackbusch17}, the solution $y(\cdot,\omega)$ is continuously differentiable. Failure occurs if the flow rate 
\begin{align}
q(x,\omega) := -a_{\mathrm{I}/\mathrm{II},n}(x,\omega) \frac{\partial y(x,\omega)}{\partial x}\label{flow rate}
\end{align}
is larger than $q_{\mathrm{max}}$ at $x=1$. Hence, the operator $\mathcal{F}$ is given as the point evaluation of the flow rate $q$ at $x=1$, which yields the LSF $G(U(\omega)) =q_{\mathrm{max}} - q(1,\omega)$. Linear FEs are applied to derive a discretization. By \cite[Section 1.6]{Strang97}, it follows that 
\begin{align*}
\Vert y(\cdot, \omega) - y_h(\cdot, \omega)\Vert_{W^{1,\infty}} \le C \Vert y(\cdot,\omega)\Vert_{W^{2,\infty}} h.
\end{align*} 
For the diffusion coefficient $a_{\mathrm{I}}$, $C \Vert y(\cdot,\omega)\Vert_{W^{2,\infty}}\le C_{\mathrm{FE}}$ can be uniformly bounded and Assumption~\ref{Ass LSF error} is satisfied for $s=1$. However, for the diffusion coefficient $a_{\mathrm{II}}$, $C_{\mathrm{FE}}(\omega)=C \Vert y(\cdot,\omega)\Vert_{W^{2,\infty}}$ is a random variable. In Section~\ref{Section 5.3}, we consider again this example with the log-normal diffusion coefficient $a_{\mathrm{II}}$.
\end{example}

\subsection{FORM probability of failure} 
We derive an upper bound for the error given in~\eqref{Err_proboffailure} which depends on the PDE discretization error and on an approximation of the probability of failure. This approximation will be given by the FORM estimate of the probability of failure; see \cite{Hasofer74,Kiureghian04} for details. 
We now briefly introduce the FORM method.

We define the MLFP $u^{\mathrm{MLFP}} \in \mathbb{R}^n$ as the solution of the minimization problem
\begin{align*}
\underset{u\in\mathbb{R}^n}{\mathrm{min}} \quad \frac{1}{2} \Vert u \Vert_2^2, \quad\text{such that}\quad G(u) = 0.
\end{align*}
Hence, $u^{\mathrm{MLFP}}$ is the element of the set $\{ G = 0\}$ that has smallest distance to the origin and, thus, maximizes the Gaussian density $\varphi_n$.
Accordingly, we denote the MLFP with respect to the discretization $G_h$ as $u_h^{\mathrm{MLFP}}$. We require that $G(0)>0$ and $G_h(0)>0$, since we are generally interested in estimating failure probabilities which are in the tail of the densities. Using the MLFPs, we obtain an estimate for the probability of failure via
\begin{align*}
P_f^{\mathrm{FORM}} = \Phi(-\Vert u^{\mathrm{MLFP}}\Vert_2) \quad\text{and}\quad P_{f,h}^{\mathrm{FORM}} = \Phi\left(-\Vert u_h^{\mathrm{MLFP}}\Vert_2\right),
\end{align*}
where $\Phi$ is the CDF of the one-dimensional standard normal distribution. The FORM estimate is equal to the probability mass of the half-space which is defined through the hyperplane at the MLFP with direction perpendicular to the surface of the failure domain. Thus, the FORM estimate is an upper bound for the probability of failure, if the failure domain is convex. We use the convexity of the failure domains as an assumption on the LSF for Theorem~\ref{Thm error bound}. We state the convexity assumption in Assumption \ref{Ass. convex failure domains}. In Assumption~\ref{Assumption Gradient LSF}, we state an assumption on the gradient of the LSF, which is relevant for the proof of Theorem~\ref{Thm. LSF distance}.

\begin{assumption}[Geometry of the failure domains]\label{Ass. convex failure domains}
The failure domains $A$ and $A_h$ are unbounded, convex sets. 
\end{assumption}
\begin{assumption}[Non-degeneracy of $\nabla_u G$ and $\nabla_u G_h$ at the limit-state surface]\label{Assumption Gradient LSF}
For all $h>0$ there exists $\nu_h>0$ such that for almost every $u\in\partial A$ it holds $\nabla_u G(u)\neq 0$, $\nabla_u G_h(u)\neq 0$ and $\vert\cos\left(\sphericalangle\left(u-u_h,\nabla_u G(u)\right)\right)\vert\ge\nu_h$, where $u_h\in\partial A_h$ is the point that has minimal distance to $u$ and $\sphericalangle(\cdot,\cdot)$ denotes the angle between two vectors. 
\end{assumption}
Assumption~\ref{Assumption Gradient LSF} states that the direction from a point $u\in \partial A$ to its nearest neighbour $u_h\in\partial A_h$ is not orthogonal to the gradient $\nabla_u G(u)$. 

\subsection{Error bound for the probability of failure}\label{subsection error bound}
The following proposition and theorem are the main statements of this manuscript. Proposition~\ref{proposition} states an error bound of the relative error with respect to the FORM estimates $P_f^{\mathrm{FORM}}$ and $P_{f,h}^{\mathrm{FORM}}$. This bound is applicable in the general case where the geometries of the failure domains are unknown. In particular, the convexity of the failure domains is not required. In Theorem~\ref{Thm error bound}, we require convexity of the failure domains. Hence, the situation is more restrictive as compared with Proposition~\ref{proposition}. In this case, we derive an error bound of the absolute error $\vert P_f -P_{f,h}\vert$ in dependence of the discretized FORM estimate $P_{f,h}^{\mathrm{FORM}}$. Subsequent to Theorem~\ref{Thm error bound}, we give a remark which discusses the error bounds if the approximation error of $G$ and $G_h$ in Assumption~\ref{Ass LSF error} is not uniformly bounded. This remark builds on Remark~\ref{Remark non-uniformly elliptic} and~\ref{Remark uniform}.

\begin{proposition}\label{proposition}
Let $a(x,\omega)$ be a uniformly elliptic and bounded diffusion coefficient and let Assumptions~\ref{regularity of a},~\ref{Ass gaussian},~\ref{Ass LSF error},~\ref{Ass lipschitz}, and~\ref{Assumption Gradient LSF} hold. Then for $h>0$ sufficiently small, the relative error of the FORM estimates is upper bounded by 
\begin{align}
\frac{\vert P_f^{\mathrm{FORM}}-P_{f,h}^{\mathrm{FORM}}\vert}{P_f^{\mathrm{FORM}}} \le \widehat{C}^{\mathrm{FORM}} h^s.\label{error bound proposition}
\end{align}
\end{proposition}

\begin{theorem}\label{Thm error bound}
Let $a(x,\omega)$ be a uniformly elliptic and bounded diffusion coefficient and let Assumptions~\ref{regularity of a},~\ref{Ass gaussian},~\ref{Ass LSF error},~\ref{Ass lipschitz},~\ref{Ass. convex failure domains} and~\ref{Assumption Gradient LSF} hold. Then for $h>0$ sufficiently small, the error of the exact and approximate probability of failure is upper bounded by
\begin{align}
\vert P_f - P_{f,h}\vert \le \widehat{C} h^s P_{f,h}^{\mathrm{FORM}}.\label{error bound theorem}
\end{align}
\end{theorem}

We note that the constants $\widehat{C}^{\mathrm{FORM}}>0$ and $\widehat{C}>0$ in~\eqref{error bound proposition} and~\eqref{error bound theorem}, respectively, depend on $C_{\mathrm{FE}}$, $h$, $n$, $\Vert u^{\mathrm{MLFP}}\Vert_2$, and $\Vert u_h^{\mathrm{MLFP}}\Vert_2$. We will discuss the behaviour of $\widehat{C}^{\mathrm{FORM}}$ and $\widehat{C}$ with respect to their dependencies in the following sections. 
The outline of the proof of Theorem~\ref{Thm error bound} is as follows:
\begin{itemize}
\item[(P1)] For the absolute error, we derive the bound
\begin{align*}
\vert P_f - P_{f,h}\vert \le C_1 h^s,
\end{align*}
where $C_1$ depends on the local Lipschitz constant $C_L$ of the CDF of $G(U)$.
\item[(P2)] Under the assumption that $G$ is affine linear with respect to $U$, we derive an upper bound for the local Lipschitz constant of the CDF of $G(U)$ which depends linearly on $P_{f}$. This yields an upper bound for the relative error
\begin{align*}
\vert P_f - P_{f,h}\vert/P_{f} \le C_2 h^s.
\end{align*}
\item[(P3)] For LSFs of the form~\eqref{LSF form} and~\eqref{LSF form discrete}, we prove that the distance between the exact \emph{limit-state surface} $\partial A:=\{u\in\mathbb{R}^n:G(u)=0\}$ and its approximation $\partial A_h:=\{u\in\mathbb{R}^n:G_h(u)=0\}$ has order $\mathcal{O}(h^s)$ of convergence for $h>0$ sufficiently small, i.e., for all $u\in \partial A$ it holds
\begin{align*}
\mathrm{dist}(u, \partial A_h) \le C_3 h^s.
\end{align*}
\item[(P4)] Under the assumption that $\mathrm{dist}(u, \partial A_h) \le C_3 h^s$, we derive an upper bound for the Gaussian measure of the symmetric difference $A\triangle A_h$ in the form
\begin{align*}
\mathbb{P}[A\triangle A_h] \le C_4\mathbb{P}[U_1 \in ]-b_h-C_3h^s, -b_h + C_3h^s]],
\end{align*}
where $U_1$ is distributed according to $\mathrm{N}(0,1)$ and $b_h :=\Vert u_h^{\mathrm{MLFP}}\Vert_2$. 
\item[(P5)] We define the affine linear function $\widetilde{G}:=U_1+b_h$ and apply the derived bound of the linear case (P2) to $\widetilde{G}$ to prove~\eqref{error bound theorem}.
\end{itemize}
In the following sections, we provide full details of the steps (P1)--(P5). The proof of Proposition~\ref{proposition} requires (P1)--(P3) and a similar form of (P5) but does not require the upper bound of the Gaussian measure of $A\triangle A_h$ in (P4). Indeed, convexity of the failure domain is only required to prove (P4).

\begin{remark}\label{Remark error bound}
Given the idea from Remark~\ref{Ass LSF error}, we conjecture that a similar error bound as in~\eqref{error bound theorem} also holds in the case of pathwise ellipticity. Given an error tolerance $\epsilon>0$, we choose $R$ such that $\mathbb{P}[U\not\in B_R]\le\epsilon/2$ and we bound the random variable $C_{\mathrm{FE}}(\omega)$ in $B_R$. We propose to choose $\epsilon\ll P_f$ to ensure that the truncated tails do not contain a large probability mass of the failure domain. We define the quantities $P_f^{\epsilon}:=\mathbb{P}[G(U)\le 0 \cap U\in B_R]$ and $P_{f,h}^{\epsilon}:=\mathbb{P}[G_h(U)\le 0 \cap U\in B_R]$. By the triangle inequality it holds that
\begin{align*}
\vert P_f - P_{f,h}\vert \le \vert P_f- P_f^{\epsilon}\vert + \vert P_f^{\epsilon}-P_{f,h}^{\epsilon}\vert + \vert P_{f,h}-P_{f,h}^{\epsilon}\vert \le \vert P_f^{\epsilon}-P_{f,h}^{\epsilon}\vert + \epsilon.
\end{align*}
Thus, the restriction to $B_R$ leads to an $\epsilon$-error for the absolute error of the probability of failure estimates. The proof of Theorem~\ref{Thm error bound} can be used as a starting point to derive a similar error bound for the absolute error $\vert P_f^{\epsilon}-P_{f,h}^{\epsilon}\vert$. The same idea can be used to derive an error bound for the relative error with respect to the FORM estimates in Proposition~\ref{proposition} in the case of pathwise ellipticity. However, providing a complete proof of these bounds is out of the scope of this paper. We note that $\epsilon$ and $R$ are chosen with respect to the probability of failure $P_f$ and do not depend on $C_{\mathrm{FE}}(\omega)$. Choosing a small $\epsilon$, requires a large radius $R$ and a large upper bound for $C_{\mathrm{FE}}(\omega)$ within $B_R$. If the user specified error tolerance $\epsilon$ is chosen, the upper bound of $C_{\mathrm{FE}}(\omega)$ is constant for the whole analysis and does not blow up.
\\We conclude that the error bounds are useful also in cases where the diffusion coefficient is only pathwise elliptic and bounded, i.e., satisfies~\eqref{Def pathwise}. In the numerical experiments, we will only consider such settings. Indeed, the numerical results give evidence for our conjecture.
\end{remark}

\begin{remark}\label{Remark distance}
We note that the approximation property of the LSF given in Assumption~\ref{Ass LSF error} determines the approximation property of the probability of failure. If the approximation error of the LSF behaves in a more general form, (P1) and (P2) can be directly applied to show that the approximation error of the probability of failure behaves in the same manner. However, (P3) is only applicable for LSFs stemming from an elliptic PDE and satisfying the regularity assumptions. Indeed, if it is possible to show that (P3) holds for more general approximation properties of the LSF, then (P4) and (P5) are directly applicable. 
\\In Section~\ref{Sec 5.1}, we consider an LSF which involves an ordinary differential equation (ODE). For this example, we show that the distance of the failure domains behaves as the convergence order of the applied time stepping scheme. Thus, (P3) is also valid and our error bounds are applicable in this setting. 
\end{remark}

\subsection{Absolute error bound}\label{Sec. absolute error}
Under Assumption~\ref{Ass LSF error} and~\ref{Ass lipschitz}, we prove that the upper bound for the absolute error of the probability of failure behaves as the approximation error of the LSF, which proves (P1). This result and proof technique are similar to \cite[Lemma 3.4]{Elfverson16}. Considering equations~\eqref{probability of failure} and~\eqref{approx probability of failure}, the approximation error is based on the symmetric difference $A\triangle A_h$. The following lemma gives an upper bound for the absolute error.

\begin{lemma}\label{lemma absolute error}
Let Assumptions~\ref{Ass LSF error} and~\ref{Ass lipschitz} hold. Then, the absolute approximation error of the probability of failure is bounded in the following way:
\begin{align}
\vert P_f - P_{f,h}\vert &\le \mathbb{P}[G(U)\in ]-C_{\mathrm{FE}}h^s,C_{\mathrm{FE}}h^s]]\label{approx error}
\\ &\le 2C_LC_{\mathrm{FE}}h^s =: C_1h^s.\label{approx error lipschitz}
\end{align}
\end{lemma}
\begin{proof}
Inserting the definitions of $P_f$ and $P_{f,h}$ given in~\eqref{probability of failure} and~\eqref{approx probability of failure} in the left hand side of~\eqref{approx error} we get
\begin{align}
\vert P_f - P_{f,h} \vert &= \left\lvert \int_{\mathbb{R}^n}\left( I(G(u)\le 0) - I(G_h(u)\le 0)\right) \varphi_n(u) \mathrm{d}u \right\rvert\notag
\\ & = \left\lvert \int_{\mathbb{R}^n} \left(I(G(u)\le 0 \wedge G_h(u)> 0) - I(G(u)> 0\wedge G_h(u)\le 0)\right)\varphi_n(u) \mathrm{d}u \right\rvert\notag
\\ &\le \mathbb{P}[\{G(U)\le 0\} \cap \{G_h(U)>0\}] + \mathbb{P}[\{G(U)>0\} \cap \{G_h(U)\le 0\}]\label{approx error inequality}
\\&=\mathbb{P}[A\triangle A_h],\notag
\end{align} 
where~\eqref{approx error inequality} follows from the triangle inequality. Using Assumption~\ref{Ass LSF error}, we know that the case $(G(u)\le 0 \wedge G_h(u)> 0)$ only occurs if $G(u) \in ]-C_{\mathrm{FE}}h^s, 0] \cap G_h(u) \in ]0,C_{\mathrm{FE}}h^s]$. Similarly, the case $(G_h(u)\le 0 \wedge G(u)> 0)$ only occurs if $G_h(u) \in ]-C_{\mathrm{FE}}h^s, 0] \text{ and } G(u) \in ]0,C_{\mathrm{FE}}h^s]$. Therefore, the absolute approximation error is bounded by 
\begin{align*}
\vert P_f - P_{f,h} \vert \le &\mathbb{P}[G(U)\in ]-C_{\mathrm{FE}}h^s,0] \cap G_h(U)\in]0,C_{\mathrm{FE}}h^s]]
\\&+\mathbb{P}[G_h(U)\in ]-C_{\mathrm{FE}}h^s,0] \cap G(U)\in]0,C_{\mathrm{FE}}h^s]].
\end{align*}
Applying the multiplication rule $\mathbb{P}[B_1 \cap B_2] = \mathbb{P}[B_1\mid B_2] \mathbb{P}[B_2]$ for $B_1,B_2\in\mathcal{A}$ we get that
\begin{align*}
\vert P_f - P_{f,h} \vert &\le \mathbb{P}[G_h(U) \in ]0,C_{\mathrm{FE}}h^s] \mid G(U)\in ]-C_{\mathrm{FE}}h^s, 0]]\cdot \mathbb{P}[G(U)\in ]-C_{\mathrm{FE}}h^s, 0]]
\\ &\hspace{0.5cm} +\mathbb{P}[G_h(U) \in ]-C_{\mathrm{FE}}h^s, 0] \mid G(U)\in ]0, C_{\mathrm{FE}}h^s]]\cdot \mathbb{P}[G(U)\in ]0, C_{\mathrm{FE}}h^s]]
\\&\le \mathbb{P}\left\lbrack G(U)\in]-C_{\mathrm{FE}}h^s,C_{\mathrm{FE}}h^s]\right\rbrack,
\end{align*}
where the last step follows from the fact that probabilities are always bounded by one. This proves inequality~\eqref{approx error}. To prove~\eqref{approx error lipschitz}, we use the assumption on the local Lipschitz continuity of $\mathbb{P}[G(U) \leq \cdot]$
\begin{align*}
\mathbb{P}[G(U)\in]-C_{\mathrm{FE}}h^s,C_{\mathrm{FE}}h^s]]\le 2C_LC_{\mathrm{FE}}h^s.
\end{align*}
\end{proof}
\begin{remark}
By switching the roles of $G$ by $G_h$ in Lemma~\ref{lemma absolute error}, we obtain
\begin{align*}
\vert P_f - P_{f,h}\vert \le \mathbb{P}[G_h(U)\in ]-C_{\mathrm{FE}}h^s,C_{\mathrm{FE}}h^s]]\le 2 C_{L,h} C_{\mathrm{FE}} h^s.
\end{align*}
\end{remark}

\section{Affine linear limit-state function}\label{Sec linear lsf}
The next step of the proofs of Proposition~\ref{proposition} and Theorem~\ref{Thm error bound} is (P2). Here, we need to find an upper bound for the local Lipschitz constant $C_{L}$ of Assumption~\ref{Ass lipschitz} around the limit-state surface $\partial A$ for the case where $G$ is affine linear with respect to $U$. 

\begin{assumption}\label{Ass linear}
The LSF $G:\mathbb{R}^n\rightarrow\mathbb{R}$ is affine linear in the Gaussian random variable $U$, i.e., $G(U) = \alpha^TU +\beta$ where $\alpha\in\mathbb{R}^n$, $\beta > 0$ and $U\sim \mathrm{N}(0,\mathrm{Id}_n)$. Therefore, the probability of failure is
\begin{align*}
P_f = \mathbb{P}[G(U)\le 0] = \mathbb{P}[\alpha^TU\le -\beta] =\int_{-\infty}^{-\beta} \frac{1}{\sqrt{2\pi\Vert \alpha\Vert_2^2}} \exp\left(-\frac{u^2}{2\Vert \alpha \Vert_2^2}\right)\mathrm{d}u = F_W(-\beta),
\end{align*}
where $W:=\alpha^TU$ is distributed according to $\mathrm{N}(0,\Vert \alpha \Vert_2^2)$ and $F_W(\cdot)$ is the CDF of $W$.
\end{assumption}
The shift parameter $\beta$ is assumed to be positive which yields that $G(0)>0$. Moreover, we require that $-\beta +C_{\mathrm{FE}}h^s<0$ since the lower and upper bounds in~\eqref{CDF bound 1} are only defined for negative inputs and we evaluate these bounds at $w=-\beta +C_{\mathrm{FE}}h^s$. Together with Assumptions~\ref{Ass LSF error},~\ref{Ass lipschitz} and~\ref{Ass linear}, we prove statement (P2).

\begin{theorem}\label{relative error linear case}
Let Assumptions~\ref{Ass LSF error},~\ref{Ass lipschitz} and~\ref{Ass linear} hold and $-\beta +C_{\mathrm{FE}}h^s<0$. Then the relative approximation error of the probability of failure is bounded by
\begin{align*}
\frac{\vert P_f - P_{f,h}\vert}{P_f} \le C_2(\beta,\sigma,h^s,C_{\mathrm{FE}})\cdot h^s,
\end{align*}
where $\sigma^2 = \Vert \alpha \Vert_2^2$. 
\end{theorem}
By Assumption~\ref{Ass linear}, the probability of failure is directly given in terms of the CDF $F_W$. The goal is to derive an upper bound for the local Lipschitz constant $C_{L}$ in Assumption~\ref{Ass lipschitz}, which depends linearly on the probability of failure. By Lemma~\ref{lemma absolute error}, this is equivalent to deriving an upper bound for the local Lipschitz constant of $F_W(\cdot)$ in the interval $]-\beta -C_{\mathrm{FE}}h^s, -\beta + C_{\mathrm{FE}}h^s]$. We distinguish two cases in the proof. First, we assume that the approximation $G_h$ is one-sided, i.e., $G(U)\le G_h(U)$ almost surely. Secondly, we consider the non-one-sided case. 

\subsection{One-sided approximation}
\begin{assumption}\label{assumption nestedness}
The approximation $G_h$ of the LSF is one-sided with respect to the exact LSF $G$, that means
\begin{align*}
G(u) \le G_h(u),
\end{align*}
for all $u\in\mathbb{R}^n$ and $h >0$.
\end{assumption}
Under Assumption~\ref{assumption nestedness}, it follows that $P_{f,h}\le P_f$ and $\mathbb{P}[G_h(U) \in ]-C_{\mathrm{FE}}h^s, 0] \mid G(U)\in ]0, C_{\mathrm{FE}}h^s]]=0$. Hence, the bound for the absolute approximation error in Lemma~\ref{lemma absolute error} simplifies to
\begin{align*}
\vert P_f - P_{f,h}\vert \le \mathbb{P}[G(U)\in ]-C_{\mathrm{FE}}h^s, 0]]\le C_L C_{\mathrm{FE}} h^s.
\end{align*}
Therefore, it suffices to derive an upper bound for the local Lipschitz constant of $F_W(\cdot)$ within the interval $]-\beta - C_{\mathrm{FE}}h^s, -\beta]$. Observe that the derivative of the CDF $F_W(w)$ with respect to $w$ is the PDF of the normal distribution with mean $0$ and variance $\sigma^2$, denoted by $\varphi_W(w)=\exp\left(-w^2/(2\sigma^2)\right)/\sqrt{2\pi\sigma^2}$, which is strictly increasing for $w\in]-\infty, 0[$. Therefore, the local Lipschitz constant of $F_W$ on the interval $]-\beta-C_{\mathrm{FE}}h^s,-\beta]$ is given by $\varphi_W(-\beta)$. The goal is to derive an upper bound for $C_{L}$ in the form $C_{L}C_{\mathrm{FE}}\le C_2 P_f$ for $C_2>0$. In order to derive this result, we consider the following bounds for the CDF. 
\begin{proposition}
An upper and lower bound for the CDF $F_W(\cdot)$ on the interval $w\in]-\infty,0[$ are given by
\begin{align}
F_l(w) := -\frac{\sigma\exp\left(-\frac{w^2}{2\sigma^2}\right)}{\sqrt{2\pi}}\frac{w}{w^2+1} \le F_W(w) \le -\frac{\sigma\exp\left(-\frac{w^2}{2\sigma^2}\right)}{w\sqrt{2\pi}} =:F_u(w).\label{CDF bound 1}
\end{align}
The derivation of these bounds is given in \cite{Gordon1941}.
\end{proposition}
With these bounds, we derive an upper bound for the local Lipschitz constant $C_{L}$ having the desired form.

\begin{lemma}\label{Lipschitz bound}
Under Assumptions~\ref{Ass lipschitz},~\ref{Ass linear} and~\ref{assumption nestedness}, it holds that the local Lipschitz constant $C_{L}$ within the interval $[-\beta-C_{\mathrm{FE}}h^s,-\beta]$ is bounded by 
\begin{align*}
C_{L} \le \left(\frac{\beta}{\sigma^2}+\frac{1}{\beta}+\frac{1}{\beta\sigma^2}+\frac{1}{\beta^3}\right)P_f.
\end{align*}
\end{lemma}

\begin{proof}
1. The derivative of $F_u$ is given by
\begin{align*}
F_u'(w) = \frac{\exp\left(-\frac{w^2}{2\sigma^2}\right)}{\sqrt{2\pi}\sigma}+\frac{\sigma\exp\left(-\frac{w^2}{2\sigma^2}\right)}{w^2\sqrt{2\pi}} = \frac{\exp\left(-\frac{w^2}{2\sigma^2}\right)}{\sqrt{2\pi}\sigma}\left(1 + \frac{\sigma^2}{w^2}\right).
\end{align*}
Hence, the derivative $F_u'$ is also an upper bound for the PDF $\varphi_W(w)$. Moreover, $F_u'(w)$ is an increasing function for $w<0$ since 
\begin{align*}
F_u''(w) = \frac{-w\exp\left(-\frac{w^2}{2\sigma^2}\right)}{\sqrt{2\pi}\sigma^3} - \frac{\exp\left(-\frac{w^2}{2\sigma^2}\right)}{w\sqrt{2\pi}\sigma} - \frac{2\sigma\exp\left(-\frac{w^2}{2\sigma^2}\right)}{w^3\sqrt{2\pi}}>0, \quad\text{for all } w<0.
\end{align*}
Therefore, the derivative $F_u'$ at $w=-\beta$ gives us an upper bound for the local Lipschitz constant $C_{L}$. 

2. $F_u'$ can be written in terms of $F_u$:
\begin{align}
F_u'(-\beta) = -\frac{\sigma\exp\left(-\frac{\beta^2}{2\sigma^2}\right)}{\beta\sqrt{2\pi}}\left(-\frac{\beta}{\sigma^2} - \frac{1}{\beta}\right)=F_u(-\beta)\left(\frac{\beta}{\sigma^2} + \frac{1}{\beta}\right).\label{eq. 4.1 1}
\end{align}
Since $F_u$ is an upper bound for $F_W$, we know that $F_u(-\beta) \ge F_W(-\beta) = P_f$. Since $F_l$ is a lower bound for $F_W$, we know that $F_l(-\beta) \le P_f$. Combining these two statements, we get
\begin{align}
1\le \frac{F_u(-\beta)}{P_f} \le \frac{F_u(-\beta)}{F_l(-\beta)} = \frac{\beta^2+1}{\beta^2},\label{eq. 4.1 2}
\end{align}
which yields $F_u(-\beta)\le \left(1+1/\beta^2\right)P_f$. 

Given~\eqref{eq. 4.1 1} and~\eqref{eq. 4.1 2}, we conclude that
\begin{align*}
C_{L} = \varphi_W(-\beta) \le F_u'(-\beta) = F_u(-\beta)\left(\frac{\beta}{\sigma^2} + \frac{1}{\beta}\right)\le\left(\frac{\beta}{\sigma^2}+\frac{1}{\beta}+\frac{1}{\beta\sigma^2} + \frac{1}{\beta^3}\right)P_f.
\end{align*}
\end{proof}
Combining the statements of Lemma~\ref{lemma absolute error} and Lemma~\ref{Lipschitz bound}, we conclude the proof of Theorem~\ref{relative error linear case} with the constant 
\begin{align}
C_{2,1}(\beta,\sigma) := \left(\frac{\beta}{\sigma^2}+\frac{1}{\beta}+\frac{1}{\beta\sigma^2}+\frac{1}{\beta^3}\right)C_{\mathrm{FE}}. \label{C_2_1 constant}
\end{align}

\begin{remark}[Sharper bounds]
Applying the bounds of \cite[7.1.13]{Abramowitz64}, we can derive the following sharper upper and lower bounds of $F_W(w)$ for $w\in]-\infty,0]$
\begin{align}
\sqrt{\frac{2}{\pi}}\frac{\exp\left(-\frac{w^2}{2\sigma^2}\right)}{\sqrt{4+w^2/\sigma^2}-w/\sigma} \le F_W(w) \le \sqrt{\frac{2}{\pi}}\frac{\exp\left(-\frac{w^2}{2\sigma^2}\right)}{\sqrt{2+w^2/\sigma^2}-w/\sigma}.\label{CDF bound 2}
\end{align}
Using~\eqref{CDF bound 2}, we derive a sharper bound for the Lipschitz constant $C_{L}$ of the form
\begin{align*}
C_{L}C_{\mathrm{FE}} &\le \left(\frac{\sqrt{4+\beta^2/\sigma^2}+\beta/\sigma}{\sqrt{2+\beta^2/\sigma^2}+\beta/\sigma}\right)\left(\frac{\beta}{\sigma^2}+\frac{\beta/\sqrt{2+\beta^2/\sigma^2}+\sigma}{\sigma^2\sqrt{2+\beta^2/\sigma^2}+\beta\sigma}\right)C_{\mathrm{FE}}P_f
\\&=: \widehat{C}_{2,1}(\beta,\sigma)P_f.
\end{align*}
The proof works in a similar way as the proof of Lemma~\ref{Lipschitz bound}.
\end{remark}
Figure~\ref{Fig Lipschitz bounds} shows the constants $C_{2,1}$ and $\widehat{C}_{2,1}$ for varying $\beta$ and $\sigma^2\in\{0.1, 1.0, 10.0\}$ and $C_{\mathrm{FE}}=1$. The x-axis shows the probability of failure $P_f = F_W(-\beta)$. We observe that the constants are large for small variances $\sigma^2$. Since $F_u(0)$ is not defined and $F_l(0)=0$, the constant $C_{2,1}$ explodes to $+\infty$ for large probability of failures, i.e., $\beta\rightarrow 0$ while $\widehat{C}_{2,1}$ yields a small constant in this case. In both plots, we infer that the variance $\sigma^2$ has the main influence on the behaviour of the constants. The probability of failure has only a small influence. For the FORM estimate, the variance of the linearized LSF is the square of the norm of the gradient at the MLFP. 

\begin{figure}[htbp]
\centering
	\includegraphics[trim=0cm 0cm 0cm 0cm,scale=0.23]{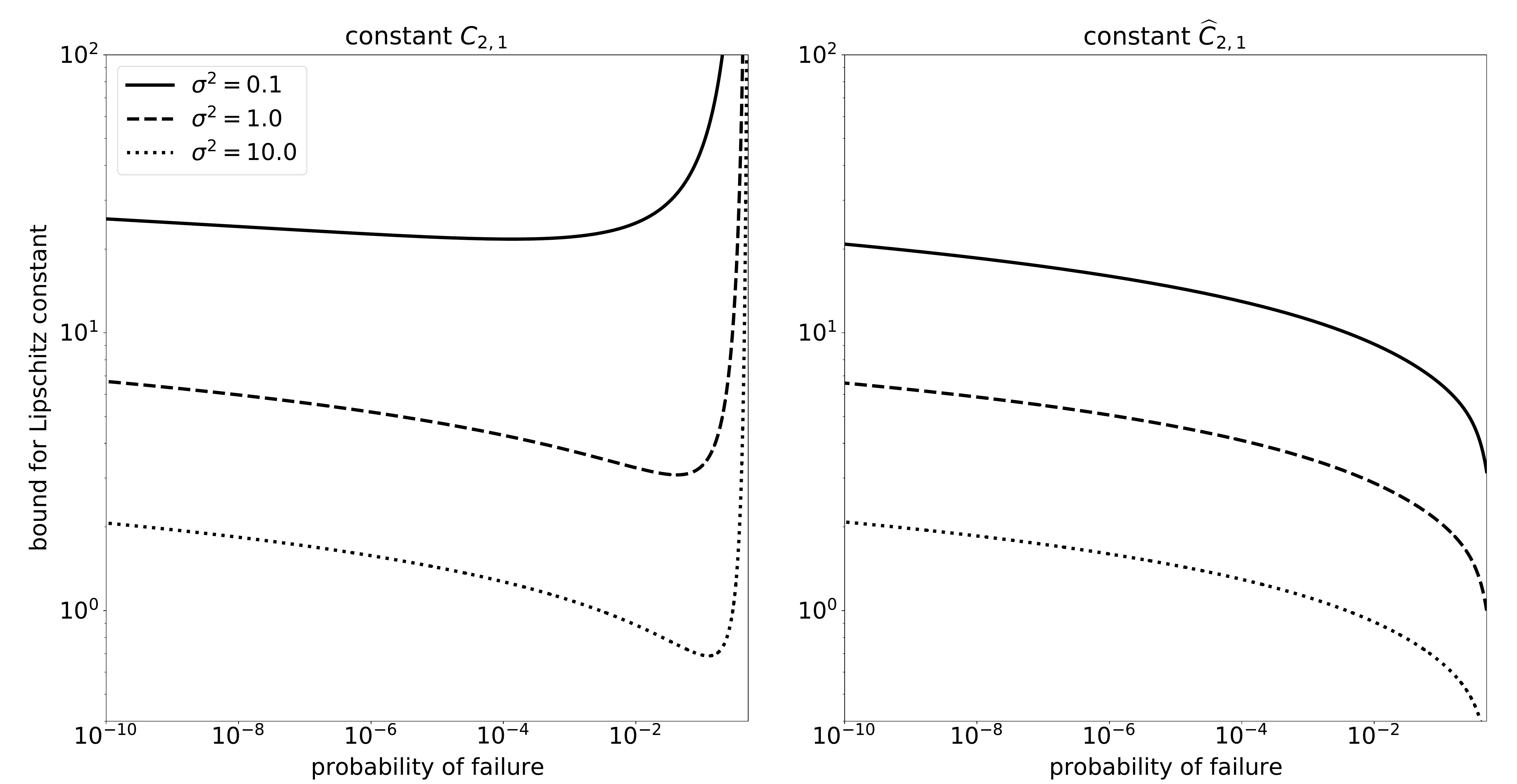}
		\caption{On the left: behaviour of $C_{2,1}$ for varying $\beta$ and $\sigma^2\in\{0.1, 1.0, 10.0\}$, where we assume that $G(U)$ is distributed according to $\mathrm{N}(\beta, \sigma^2)$ and $C_{\mathrm{FE}}=1$. On the right: behaviour of $\widehat{C}_{2,1}$ for varying $\beta$ and $\sigma^2\in\{0.1, 1.0, 10.0\}$. The x-axis shows the respective probability of failure $P_f = F_W(-\beta)$ which lies in the interval $[10^{-10}, 1/2]$.}
		\label{Fig Lipschitz bounds}
\end{figure}

\subsection{Non-one-sided approximation}
In general, we do not know if the LSF is one-sided, i.e., if $G(u)\le G_h(u)$ for all $u\in\mathbb{R}^n$. Hence, we are required to bound the local Lipschitz constant in Assumption~\ref{Ass lipschitz} within the interval $[-\beta-C_{\mathrm{FE}}h^s, -\beta + C_{\mathrm{FE}}h^s]$. Due to the fact that the bounds in~\eqref{CDF bound 1} are not defined for $w\ge0 $, $h$ must be chosen sufficiently small such that $-\beta + C_{\mathrm{FE}}h^s<0$. Since we have already derived an upper bound for the local Lipschitz constant within the first half $[-\beta-C_{\mathrm{FE}}h^s, -\beta]$, we derive an upper bound for the second half $[-\beta, -\beta+C_{\mathrm{FE}}h^s]$. In this case, the local Lipschitz constant can be expressed by the derivative of the CDF $\varphi_W(z)$ at $z=-\beta + C_{\mathrm{FE}}h^s$. Applying the same steps as for the one-sided case, yields the following bound
\begin{align}
C_{L}C_{\mathrm{FE}} &\le F_u'(-\beta+C_{\mathrm{FE}}h^s)C_{\mathrm{FE}}\notag
\\&\le \left(\beta + \frac{1}{\beta}\right)\left(\frac{1}{\sigma^2}+\frac{1}{(\beta-C_{\mathrm{FE}}h^s)^2}\right)\exp\left(\frac{2\beta C_{\mathrm{FE}}h^s-C_{\mathrm{FE}}^2h^{2s}}{2\sigma^2}\right)C_{\mathrm{FE}}P_f \notag
\\ &=: C_{2,2}(\beta,\sigma,h^s,C_{\mathrm{FE}}) P_f.\label{non-monotone bound}
\end{align}
Unfortunately, $C_{2,2}$ depends on the error bound $C_{\mathrm{FE}} h^s$ of the LSF approximation. We observe that $C_{2,2}$ converges to $C_{2,1}$ for $h\rightarrow 0$. We illustrate this in the following.
Figure~\ref{Fig Lipschitz bound non monotone} displays the constant $C_{2,2}/C_{\mathrm{FE}}$ while varying the step size $h$. We divide by $C_{\mathrm{FE}}$ to eliminate the linear dependence in $C_{2,2}$. The order of convergence is either $s=1$ or $s=2$. Moreover, different values for $C_{\mathrm{FE}}$ are considered. The variance is $\sigma^2=1$ and $\beta=4$. This yields a probability of failure of order $10^{-5}$. We observe that the constant $C_{2,2}/C_{\mathrm{FE}}$ is large for large values of $C_{\mathrm{FE}}$. For the convergence order $s=1$, the step size $h$ should be smaller than $10^{-2}$ to ensure a small constant even for large $C_{\mathrm{FE}}$. For $s=2$, the step size should be smaller than $10^{-1}$. 
\begin{figure}[htbp]
\centering
	\includegraphics[trim=0cm 0cm 0cm 0cm,scale=0.23]{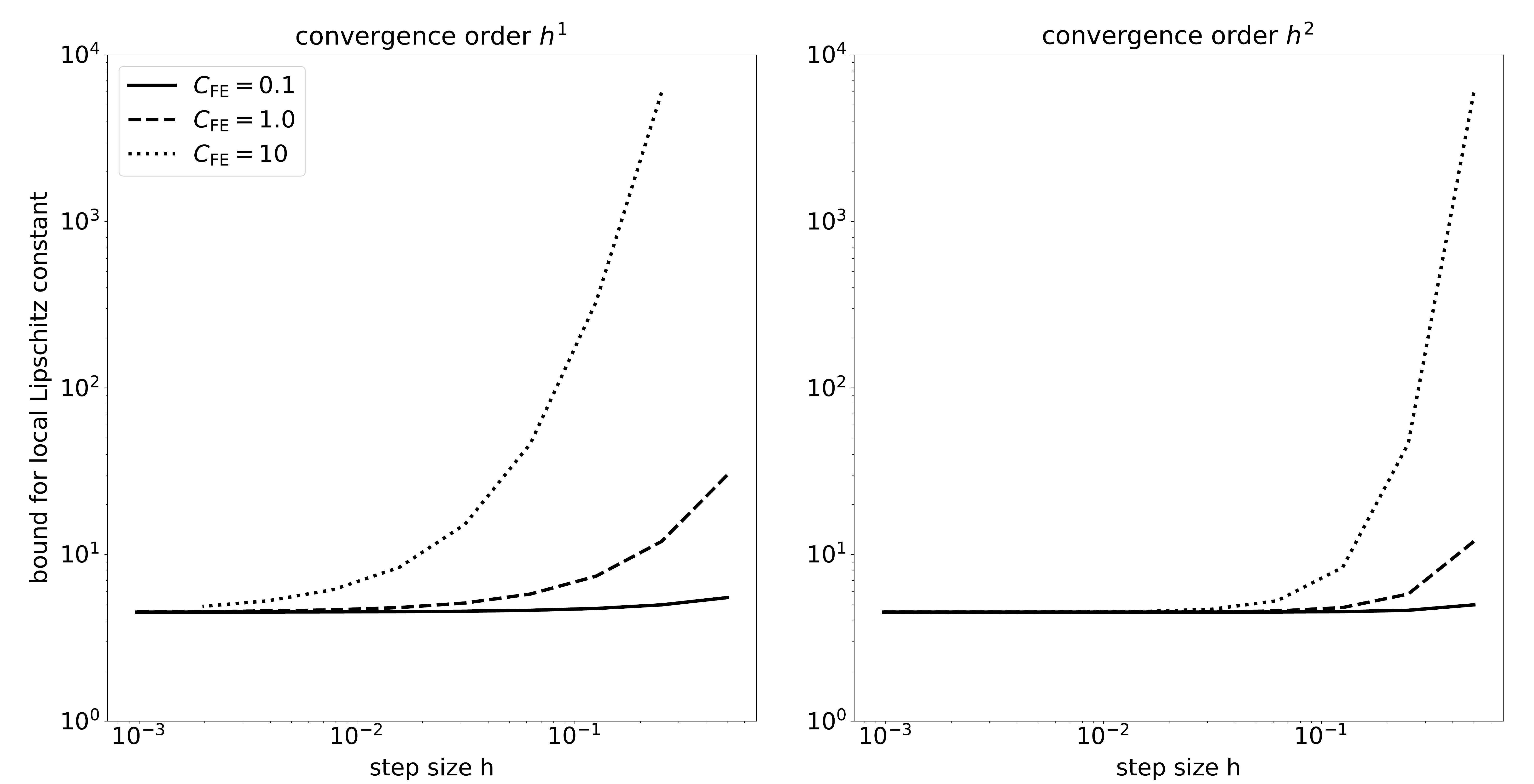}
		\caption{On the left: behaviour of $C_{2,2}/C_{\mathrm{FE}}$ for varying the step size $h$ and $C_{\mathrm{FE}}\in\{0.1, 1.0, 10.0\}$. The convergence order is $s=1$. On the right: behaviour of $C_{2,2}/C_{\mathrm{FE}}$ for varying the step size $h$ and $C_{\mathrm{FE}}\in\{0.1, 1.0, 10.0\}$. The convergence order is $s=2$. In both plots, $\sigma^2=1$ and $\beta=4$.}
		\label{Fig Lipschitz bound non monotone}
\end{figure}
\begin{remark}
Following the same steps to derive the constant $C_{2,2}$, one could also derive a constant $\widehat{C}_{2,2}$ which is based on the sharper CDF bounds~\eqref{CDF bound 2}. This constant also depends on $C_{\mathrm{FE}}h^s$. 
\end{remark}

Finally, we note that the result above allows us to prove Theorem~\ref{relative error linear case}. Combining the constants in~\eqref{C_2_1 constant} and~\eqref{non-monotone bound}, we obtain the asserted inequality with
\begin{align*}
C_2(\beta,\sigma,h^s,C_{\mathrm{FE}}):=C_{2,1}(\beta,\sigma) + C_{2,2}(\beta,\sigma,h^s,C_{\mathrm{FE}}).
\end{align*}

\section{Error Analysis with Optimal Control and FORM}\label{Sec. Opt Control}
In this section, we prove statements (P3)--(P5) which will conclude the proofs of Proposition~\ref{proposition} and Theorem~\ref{Thm error bound}. In the affine linear case, we know that the error of the LSF in Assumption~\ref{Ass LSF error} is directly related to the distance between the exact and approximate limit-state surface. However, this direct relation is not obvious in a more general setting.
\\In the subsequent step, we derive an upper bound for the distance between the exact and approximate limit-state surface $\partial A$ and $\partial A_h$, respectively, in the case where the LSF satisfies the assumptions of Proposition~\ref{proposition} and Theorem~\ref{Thm error bound}. This upper bound is based on results from optimal control theory. Finally, we estimate the Gaussian measure of the symmetric difference of the failure domains $A$ and $A_h$ by the FORM approximation $P_{f,h}^{\mathrm{FORM}}$.

\subsection{Theoretical results from optimal control}
We will formulate item (P3) in Theorem~\ref{Thm. LSF distance}. For the proof of this result, we require several concepts from the theory of optimal control. We discuss those results below; based on the works \cite{Rannacher05, Vexler04}. 

We commence with the unconstrained optimal control problem. It is given by
\begin{align}
\underset{u\in \mathbb{R}^n}{\mathrm{min}}\quad J(u) := \frac{1}{2}\Vert \widehat{G}(u)-\overline{g}\Vert_2^2 + \frac{\alpha}{2}\Vert u-\overline{u}\Vert_2^2,\label{unconstrained}
\end{align}
where $u\in \mathbb{R}^n$ is the unknown parameter and $\widehat{G}:\mathbb{R}^n\rightarrow\mathbb{R}^{n_m}$ is the observation operator. The observation operator $\widehat{G}(u) = \mathcal{B}y(\cdot,u)$ implicitly depends on the solution $y(\cdot,u)\in V=H_0^1(D)$ of the weak formulation of an elliptic PDE
\begin{align}
\int_D a(x,u) \nabla_x y(x,u)\cdot\nabla_x v(x)\mathrm{d}x = \int_D f(x) v(x)\mathrm{d}x\quad\forall v\in V,\label{weak form opt control}
\end{align} 
where $u\in \mathbb{R}^n$ is a fixed parameter, $f\in L^2(D)$, and $D\subset\mathbb{R}^d$, $d=1,2,3$, is an open, bounded, convex polygonal domain. The function $a(x,\cdot):\mathbb{R}^n\rightarrow\mathbb{R}$ is assumed to be three times continuously differentiable for all $x\in D$ and $a(\cdot, u)\in L^{\infty}(D)$ for all $u\in\mathbb{R}^n$. The discretized unconstrained optimal control problem is denoted as
\begin{align}
\underset{u\in \mathbb{R}^n}{\mathrm{min}}\quad J_h(u) := \frac{1}{2}\Vert \widehat{G}_h(u)-\overline{g}\Vert_2^2 + \frac{\alpha}{2}\Vert u-\overline{u}\Vert_2^2,\label{discrete opt control problem}
\end{align}
where $\widehat{G}_h(u) = \mathcal{B}_hy_h(\cdot,u)$ is the discretized observation operator. Additionally, we assume that $\mathcal{B}$ and $\mathcal{B}_h$ are linear and bounded with respect to $y$ and $y_h$, respectively. The parameters $\alpha\ge 0$ and $\overline{u}$ are regularizing parameters, while $\overline{g}\in\mathbb{R}^{n_m}$ is a given vector of, e.g., measurements. 

In the following, we derive the necessary and sufficient optimality conditions for~\eqref{unconstrained}. The \emph{first-order necessary optimality condition} is given as 
\begin{align*}
\nabla_u J(u) = 0 \iff (\mathcal{D}\widehat{G}(u))^T \widehat{G}(u)+\alpha u = (\mathcal{D}\widehat{G}(u))^T\overline{g}+\alpha\overline{u},
\end{align*}
where $\mathcal{D}\widehat{G}$ is the Jacobian of the observation operator $\widehat{G}$ with respect to $u$. 
The \emph{second-order optimality condition} is satisfied, if some \emph{coercivity parameter} $\gamma > 0$ exists, with
\begin{align*}
z^T \nabla_u^2 J(u^*) z\ge \gamma\Vert z\Vert_2^2\quad \forall z\in\mathbb{R}^n,
\end{align*}
where $\nabla_u^2 J(u^*)$ denotes the Hessian of $J$ with respect to $u$. 
We call $u^*\in \mathbb{R}^n$ a \emph{stable solution of}~\eqref{unconstrained}, if it satisfies both, the first and second-order optimality conditions.

\begin{theorem}\label{Thm opt control a priori}
Let Assumption~\ref{regularity of a}, ~\ref{Ass gaussian} (ii) and~\ref{Ass LSF error} hold for the weak formulation~\eqref{weak form opt control} and the operator $\mathcal{B}$. Let $u^*\in \mathbb{R}^n$ be a stable solution of~\eqref{unconstrained} such that $\nabla_u^2 J(u)$ is coercive with parameter $\gamma>0$. Then for $h>0$ sufficiently small, there exists a stable solution $u_h^*\in\mathbb{R}^n$ of the discrete problem~\eqref{discrete opt control problem} such that the following a priori error estimate holds
\begin{align*}
\Vert u^* - u_h^*\Vert_2 \le \frac{1}{\gamma} \Vert \nabla_u J(u^*)- \nabla_u J_h(u^*)\Vert_2.
\end{align*} 
\end{theorem}
\begin{proof}
The proof is given in \cite[Theorem 3.4.1 and 3.4.2]{Vexler04}.
\end{proof}

\subsection{Back to rare event estimation}
We now apply Theorem~\ref{Thm opt control a priori} to prove statement (P3).

\begin{theorem}\label{Thm. LSF distance}
Let Assumptions~\ref{regularity of a},~\ref{Ass gaussian} (ii),~\ref{Ass LSF error} and~\ref{Assumption Gradient LSF} hold. Then for all $u\in\partial A$ it holds that
\begin{align*}
\mathrm{dist}(u, \partial A_h) \le C_3(h)\cdot h^s,
\end{align*}
for $h>0$ sufficiently small.
\end{theorem}

\begin{proof}
We apply Theorem~\ref{Thm opt control a priori} to an appropriate optimal control problem with $\alpha >0$ to show that for $u^*\in\partial A$ exists a $\widehat{u}_h\in\mathbb{R}^n$ such that the distance between these points behaves as $\mathcal{O}(h^s)$. Then, we consider the limit $\alpha \rightarrow 0$.
\\Let $u^*\in\partial A$ be a point on the exact limit-state surface, i.e., $G(u^*)=0$. We investigate the following optimal control problem
\begin{align}
\underset{u\in\mathbb{R}^n}{\mathrm{min}}\quad J_{\alpha}(u):= \frac{1}{2} G(u)^2 + \frac{\alpha}{2}\Vert u-u^*\Vert_2^2,\label{surface opt control}
\end{align}
where we set $\overline{g}=0$ and $\alpha >0$. For $\alpha=0$, there is no stable solution since for all $u\in \partial A$ it holds that $J_{0}(u)=0$. Hence, the second order sufficient condition is violated. For $\alpha >0$, the gradient and Hessian matrix of $J_{\alpha}$ are given as
\begin{align*}
\nabla_u J_{\alpha}(u) &= G(u)\nabla_u G(u) + \alpha(u- u^*),
\\\nabla_u^2 J_{\alpha}(u) &= G(u)\nabla_u^2G(u) + \nabla_u G(u)(\nabla_u G(u))^{T} + \alpha I.
\end{align*}
The point $u^*$ is the unique global minimizer of~\eqref{surface opt control} since $J_{\alpha}(u^*)=0$ and $J_{\alpha}(u)>0$ for all $u\in\mathbb{R}^n\setminus \{u^*\}$. Hence, $u^*$ is a stable solution. For the discretization parameter $h>0$, we define the discretized version of~\eqref{surface opt control} as
\begin{align}
\underset{u\in\mathbb{R}^n}{\mathrm{min}}\quad J_{h,\alpha}(u):= \frac{1}{2} G_h(u)^2 + \frac{\alpha}{2}\Vert u- u^*\Vert_2^2.\label{surface discretized opt control}
\end{align}
By Theorem~\ref{Thm opt control a priori}, there exists a point $\widehat{u}_h$ in a neighborhood of $u^*$ such that $\widehat{u}_h$ is a stable solution of~\eqref{surface discretized opt control}. Thus, we know that $\nabla_u J_{h,\alpha}(\widehat{u}_h) = 0$, which yields
\begin{align*}
G_h(\widehat{u}_h)\nabla_u G_h(\widehat{u}_h) + \alpha(\widehat{u}_h-u^*)=0.
\end{align*}
If $G_h(\widehat{u}_h)=0$ we get that $\widehat{u}_h=u^*$ and the claim is proved. Now, we consider the case $G_h(\widehat{u}_h)\neq 0$ and we denote $u_h^*\in\partial A_h$ as the point on $\partial A_h$ that has minimal distance to the point $u^*$. Moreover, we define the set of points
\begin{align*}
E := \left\lbrace p\in\mathbb{R}^n: \vert\cos\left(\sphericalangle(p, \nabla_u G(u^*))\right)\vert \ge\vert\cos\left(\sphericalangle(u^*-\widehat{u}_h, \nabla_u G(u^*))\right)\vert\right\rbrace,
\end{align*} 
where $\sphericalangle(a, b) := \arccos\left({a^Tb}/{(\Vert a \Vert_2\Vert b\Vert_2)}\right)$, for $a, b \in \mathbb{R}^n$. The set $E$ contains all directions $p\in\mathbb{R}^n$, which admit a smaller or equal angle with $\nabla_u G(u^*)$ than the direction pointing from $u^*$ to $\widehat{u}_h$. For all $p \in E$, we conclude that
\begin{align*}
p^T \nabla_u^2 J_{\alpha}(u^*) p &= p^T \nabla_u G(u^*) (\nabla_u G(u^*))^T p + \alpha \Vert p \Vert_2^2 
\\&\ge \left(\cos^2\left(\sphericalangle(u^*-\widehat{u}_h, \nabla_u G(u^*))\right)\Vert\nabla_u G(u^*)\Vert_2^2 + \alpha\right)\Vert p\Vert_2^2.
\end{align*} 
Thus, the Hessian matrix $\nabla_u^2 J_{\alpha}(u^*)$ is coercive for all $p\in E$ with parameter 
\begin{align*}
\gamma = \cos^2\left(\sphericalangle(u^*-\widehat{u}_h, \nabla_u G(u^*))\right)\Vert\nabla_u G(u^*)\Vert_2^2 + \alpha.
\end{align*}
Since $u^*-\widehat{u}\in E$, we can perform the proof of Theorem~\ref{Thm opt control a priori} only for directions $p\in E$ which yields that
\begin{align*}
\Vert u^* - \widehat{u}_h\Vert_2 &\le \frac{\Vert\nabla_u J_{\alpha}(u^*)-\nabla_u J_{h,\alpha}(u^*)\Vert_2 }{\cos^2\left(\sphericalangle(u^*-\widehat{u}_h, \nabla_u G(u^*))\right)\Vert\nabla_u G(u^*)\Vert_2^2 + \alpha}
\\ & \le \frac{\Vert\nabla_u G_h(u^*)\Vert_2}{\cos^2\left(\sphericalangle(u^*-\widehat{u}_h, \nabla_u G(u^*))\right)\Vert\nabla_u G(u^*)\Vert_2^2 + \alpha}C_{\mathrm{FE}}h^s,
\end{align*}
since it holds $\Vert\nabla_u J_{h,\alpha}(u^*)\Vert_2=\vert G_h(u^*)\vert\Vert\nabla_u G_h(u^*)\Vert_2\le C_{\mathrm{FE}} h^s \Vert\nabla_u G_h(u^*)\Vert_2$ and $\nabla_u J_\alpha(u^*) = 0$. Applying the limit, $\alpha\rightarrow 0$ we conclude that $G_h(\widehat{u}_h)\rightarrow 0$ and, thus, $\widehat{u}_h\rightarrow u_h^*\in\partial A_h$. By Assumption~\ref{Assumption Gradient LSF}, it holds that $\cos^2\left(\sphericalangle\left(u^*-u_h^*, \nabla_u G(u^*)\right)\right)\ge \nu_h^2$ and we conclude that
\begin{align*}
\Vert u^* - u_h^*\Vert\le \frac{\Vert\nabla_u G_h(u^*)\Vert_2}{\nu_h^2\Vert\nabla_u G(u^*)\Vert_2^2}C_{\mathrm{FE}}h^s.
\end{align*}
Since this holds true for all $u^* \in \partial A$, we define 
\begin{align*}
C_3(h) := C_{\mathrm{FE}}\cdot \underset{u\in \partial A}{\sup}\quad \frac{\Vert\nabla_u G_h(u)\Vert_2}{\nu_h^2\Vert\nabla_u G(u)\Vert_2^2}
\end{align*}
and the desired statement is proved.
\end{proof}

\begin{remark}
For the limit $h\rightarrow 0$, it holds that $G_h(u)\rightarrow G(u)$ for all $u\in\mathbb{R}^n$. Therefore, the limit-state surface $\partial A_h$ converges to $\partial A$ and $u_h\rightarrow u$. Hence, it holds $\vert \cos\left(\sphericalangle(u-u_h,\nabla_u G(u))\right)\vert \rightarrow 1$, where $u_h\in\partial A_h$ is the point on $\partial A_h$ that has smallest distance to $u\in\partial A$. Thus, $\nu_h\rightarrow 1$, as $h\rightarrow 0$. If in addition $\nabla_u G_h(u)$ converges uniformly to $\nabla_u G(u)$, it yields that 
\begin{align*}
\underset{h\rightarrow 0}{\lim}\quad C_3(h) = C_{\mathrm{FE}}\cdot\underset{u\in\partial A}{\sup}\quad 1/\Vert \nabla_u G(u)\Vert_2.
\end{align*}
\end{remark}

With (P1)--(P3), we can now give the proof of Proposition~\ref{proposition}.
\begin{proof}[Proof of Proposition~\ref{proposition}]
We denote the distances of the MLFPs to the origin as $b=\Vert u^{\mathrm{MLFP}}\Vert_2$ and $b_h=\Vert u_h^{\mathrm{MLFP}}\Vert_2$. By definition, we know that $P_f^{\mathrm{FORM}} = \Phi(-b) = \mathbb{P}[U_1 \le -b]$, where $U_1$ is a one-dimensional random variable distributed according to $\mathrm{N}(0, 1)$. Similar, $P_{f,h}^{\mathrm{FORM}}=\Phi(-b_h)=\mathbb{P}[U_1\le -b_h]$. From Theorem~\ref{Thm. LSF distance}, we know that the distance between $\partial A$ and $\partial A_h$ is bounded from above by $C_3(h)h^s$. Thus, $\vert b - b_h\vert \le C_3(h)h^s$. This yields that the absolute error $\vert P_f^{\mathrm{FORM}} - P_{f,h}^{\mathrm{FORM}}\vert$ is bounded from above by 
\begin{align}
\vert P_f^{\mathrm{FORM}} -P_{f,h}^{\mathrm{FORM}}\vert &= \vert \mathbb{P}[U_1\le -b]-\mathbb{P}[U_1\le -b_h]\vert\label{proof part 0 proposition}
\\&\le \mathbb{P}[U_1\in ]-b-C_3(h)h^s,-b + C_3(h)h^s]],\label{proof part 1 proposition}
\end{align} 
where we apply similar steps as in the proof of Lemma~\ref{lemma absolute error}. The probability term in~\eqref{proof part 1 proposition} is equal to $\mathbb{P}[\widetilde{G}(U_1) \in ]-C_3(h)h^s,C_3(h)h^s]]$ where the LSF $\widetilde{G}(U_1):= U_1 +b$ satisfies~\eqref{ass uniform bound} in Assumption~\ref{Ass LSF error} with $C_{\mathrm{FE}}=C_3(h)$. By definition it holds $P_{f}^{\mathrm{FORM}} = \mathbb{P}[\widetilde{G}(U_1)\le 0]$. Since $\widetilde{G}$ is affine linear, we apply Theorem~\ref{relative error linear case} to the LSF $\widetilde{G}$ with $\sigma=1$ and $\beta=b$ which yields 
\begin{align}
\mathbb{P}[\widetilde{G}(U_1) \in ]-C_3(h)h^s,C_3(h)h^s]] \le C_2(b,1, h^s,C_3(h)) h^s P_{f}^{\mathrm{FORM}}.\label{proof part 2 proposition}
\end{align}
Finally, combining~\eqref{proof part 1 proposition} and~\eqref{proof part 2 proposition} we conclude that
\begin{align*}
\vert P_f^{\mathrm{FORM}} -P_{f,h}^{\mathrm{FORM}}\vert \le C_2(b,1, h^s,C_3(h)) h^s P_{f}^{\mathrm{FORM}}=:\widehat{C}^{\mathrm{FORM}} h^s P_{f}^{\mathrm{FORM}}.
\end{align*}
\end{proof}

The proof of Theorem~\ref{Thm error bound} works in a similar way as the proof of Proposition~\ref{proposition}. However, the inequalities in~\eqref{proof part 0 proposition} and~\eqref{proof part 1 proposition} do not hold directly for the absolute error $\vert P_f - P_{f,h}\vert$, which is upper bounded by the Gaussian measure of the set $S:= A\triangle A_h$. The following theorem provides an upper bound of the Gaussian measure of $S$ which is similar to~\eqref{proof part 1 proposition}. In this theorem, the convexity of the failure domains is required, i.e., we assume that Assumption~\ref{Ass. convex failure domains} holds. We switch the roles of $P_f$ and $P_{f,h}$, which yields that the derived error bound depends on $P_{f,h}^{\mathrm{FORM}}$ and not on $P_f^{\mathrm{FORM}}$.

\begin{theorem}\label{Thm bound of Gaussian measure}
Let Assumption~\ref{Ass gaussian}~(i) and~\ref{Ass. convex failure domains} hold. Moreover, we assume that for all $u_h\in\partial A_h$ it holds that
\begin{align}
\mathrm{dist}(u_h,\partial A) \le \widetilde{C} h^s.\label{Thm ass distance}
\end{align}
The distance between the origin and $\partial A_h$ is denoted as $b_h$. Then, an upper bound for the Gaussian measure of the symmetric difference of $A$ and $A_h$ is given by
\begin{align*}
\mathbb{P}[U\in A\triangle A_h] \le C_4(n)\mathbb{P}[U_1\in ]-b_h-\widetilde{C}h^s, -b_h+\widetilde{C}h^s]],
\end{align*}
where $U_1$ is distributed according to $\mathrm{N}(0,1)$. The constant $C_4$ is given by
\begin{align*}
C_4(n) = 1 + \pi^{1/2}\frac{\Gamma((n+1)/2)}{\Gamma(n/2)},
\end{align*}
where $\Gamma(\cdot)$ is the Gamma function.
\end{theorem}

\begin{proof}
For $n=1$, the statement directly follows from assumption~\eqref{Thm ass distance} with $C_4(n)=1$. Consider $n>1$. By the rotation invariance of the Gaussian measure $\mathbb{P}[U \in \cdot] = \mathrm{N}(0, \mathrm{Id}_n)$, we assume, without loss of generality, that the point with smallest distance to the origin is given by $b_h^* =(0,\dots,0,-b_h)^T\in\mathbb{R}^n$, thus, $u_h^{\mathrm{MLFP}}=b_h^*$. We denote the sets of interest by $D_1:= A\triangle A_h$ and $D_2:=\{u\in\mathbb{R}^n: u_n\in ]-b_h-\widetilde{C}h^s, -b_h+\widetilde{C}h^s]\}$. 
\\First, we consider the closed ball around the origin $B_R :=\{u\in\mathbb{R}^n: \Vert u \Vert \le R\}$ and we show that $\mathbb{P}[U\in D_1\cap B_R] \le C_4(n) \mathbb{P}[U\in D_2\cap B_R]$. Afterwards, we consider the limit $R\rightarrow \infty$.
\begin{figure}[htbp]
\centering
	\includegraphics[trim=0cm 0cm 0cm 0cm,scale=0.3]{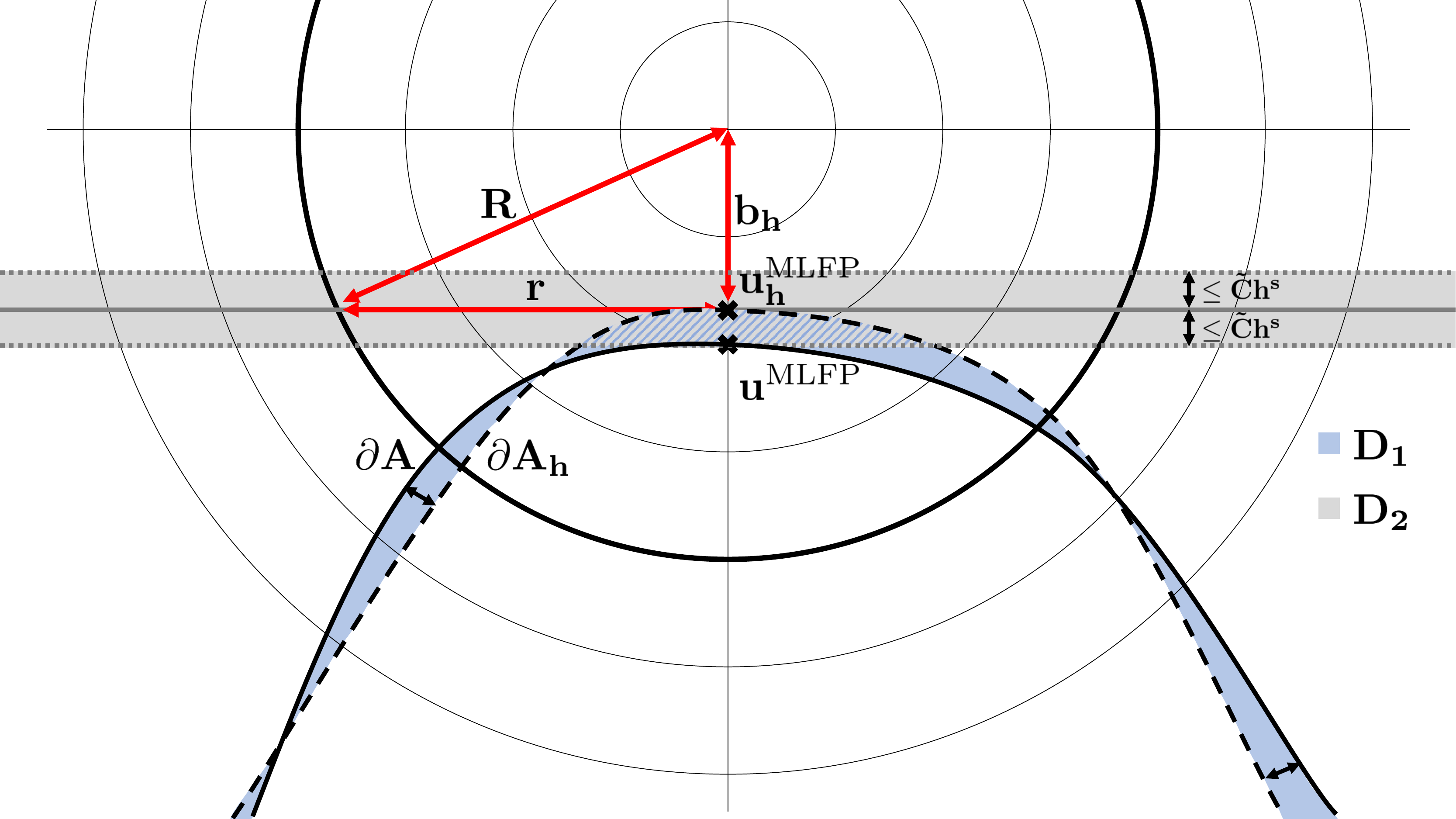}
		\caption{Illustration of the limit-state surfaces $\partial A$ and $\partial A_h$ in 2D. The illustrations shows that the Gaussian measure of $A\triangle A_h$ can be translated to a Gaussian measure at the MLFP $u_h^{MLFP}$.}
		\label{Fig proof illu}
\end{figure}
\\ For $R\le b_h$, it follows that $D_1\cap B_R=\emptyset$, and the statement is valid. Therefore, we consider $R>b_h$. First, we determine the scaling factor of the $n-1$ dimensional Lebesgue measure of the transformation of a convex curve to a hyperplane. Therefore, we consider the sets $E_1=B_R\cap \partial A_h$ and $E_2=B_R\cap\{u\in\mathbb{R}^n:u_n = -b_h\}$. Note that $E_2$ is an $n-1$-dimensional ball centred in $b^*_h$. We visualize this in Figure~\ref{Fig proof illu}. We call the radius of this ball $r :=\sqrt{R^2-b_h^2}$. Thus, the $n-1$-dimensional Lebesgue measure of $E_2$ is equal to the volume of the ball $B_r$ in $n-1$ dimensions, which is given by 
\begin{align*}
\lambda_{n-1}(E_2) = \frac{\pi^{(n-1)/2}r^{n-1}}{\Gamma((n+1)/2)}.
\end{align*}
By the convexity of $A_h$, we conclude that an upper bound for the $n-1$ dimensional Lebesgue measure of $E_1$ is given by the surface measure of the set $B_R\cap\{u\in\mathbb{R}^n:u_n\le -b_h\}$. 
This surface measure is bounded from above by the sum of the volume of the ball $B_r$ in $n-1$ dimensions and $1/2$ of the surface of the ball $B_r$ in $n$ dimensions which yields
\begin{align*}
\lambda_{n-1}(E_1) \le \frac{\pi^{(n-1)/2}r^{n-1}}{\Gamma((n+1)/2)}+\frac{\pi^{n/2}r^{n-1}}{\Gamma(n/2)}.
\end{align*}
Hence, the fraction of the two $n-1$ dimensional Lebesgue measures is bounded by
\begin{align}
\frac{\lambda_{n-1}(E_1)}{\lambda_{n-1}(E_2)}\le 1 + \pi^{1/2}\frac{\Gamma((n+1)/2)}{\Gamma(n/2)}=C_4(n).\label{Ineq scaling}
\end{align}
The formulas for the volume and surface of a ball in $n$ dimensions are given in \cite[5.19(iii)]{Nist10}. Inequality~\eqref{Ineq scaling} bounds the scaling factor of the length of the curve $E_1$ with respect to the hyperplane $E_2$. Applying this result, we can transform the set $D_1$ into $D_2$. The probability of interest is given by
\begin{align*}
\mathbb{P}[U\in D_1\cap B_R] &= \int_{u\in\mathbb{R}^n} I(u\in D_1\cap B_R)\varphi_n(u)\mathrm{d}u 
\\ &= \frac{1}{(2\pi)^{n/2}} \int_{B_R}I(u\in D_1)\exp\left(-\frac{\Vert u \Vert_2^2}{2}\right)\mathrm{d}u
\\ &= \frac{1}{(2\pi)^{n/2}} \int_{r=0}^{r=R} \int_{S_{n-1}(r)} I(u\in D_1)\exp\left(-\frac{r^2}{2}\right)\mathrm{d}s\mathrm{d}r,
\end{align*}
where $S_{n-1}(r)=\{u\in\mathbb{R}^n:\Vert u \Vert_2 = r\}$ is the surface of $B_r$. Since the distance of $A$ and $A_h$ is always smaller than $\widetilde{C}h^s$, the $n-1$ dimensional Lebesgue measure of the intersection $D_1\cap S_{n-1}(r)$ is smaller or equal than the $n-1$ dimensional Lebesgue measure of the intersection $D_2\cap S_{n-1}(r)$. Hence, since the standard Gaussian density is constant on $S_{n-1}(r)$ for all $r\ge 0$ and applying the transformation of $D_1$ to $D_2$ it follows that

\begin{align*}
\mathbb{P}[U\in D_1\cap B_R] &\le C_4(n) \frac{1}{(2\pi)^{n/2}} \int_{r=0}^{r=R} \int_{S_{n-1}(r)} I(u\in D_2)\exp\left(-\frac{r^2}{2}\right)\mathrm{d}s\mathrm{d}r
\\ &= C_4(n) \mathbb{P}[U\in D_2\cap B_R]
\\ &\le C_4(n) \mathbb{P}[U\in D_2] 
\\ &= C_4(n) \mathbb{P}[U_1\in ]-b_h-\widetilde{C}h^s,-b_h+\widetilde{C}h^s]].
\end{align*} 
Taking the limit $R\rightarrow \infty$ we get the desired statement, due to the continuity of measures.
\end{proof}

\begin{remark} 
Unfortunately, taking the limit $n\rightarrow \infty$ yields $C_4(n)\rightarrow \infty$. Thus, this result does not easily generalise with respect to infinite-dimensional settings. However, the growth of $C_4(n)$ is $\mathcal{O}(n^{1/2})$ as visualised in Figure~\ref{Fig C_4}. Hence, even in high dimensions, the constant is reasonably small.
\end{remark}

\begin{figure}[htbp]
\centering
	\includegraphics[trim=0cm 0cm 0cm 0cm,scale=0.18]{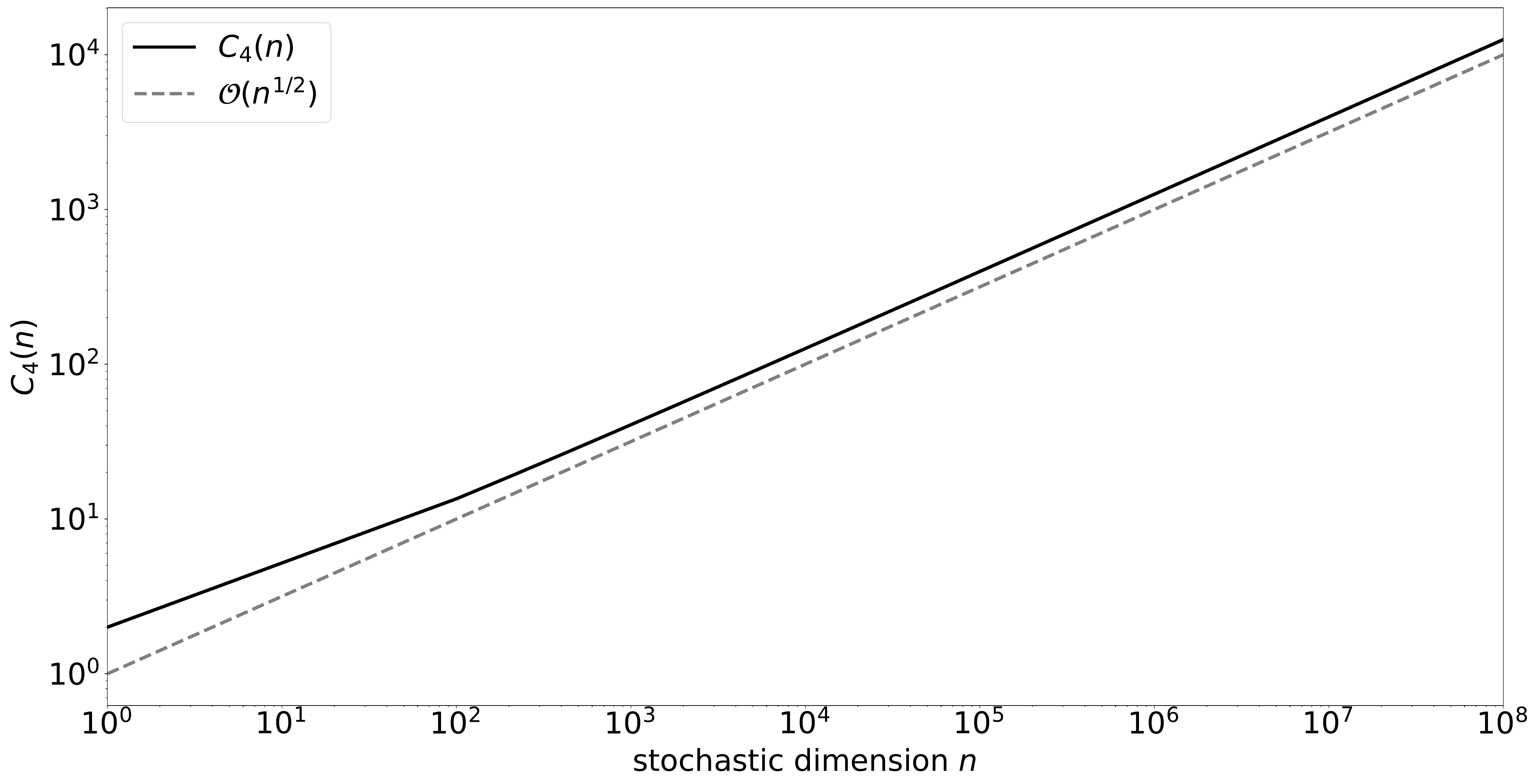}
		\caption{Behaviour of $C_{4}(n)$ for varying the stochastic dimension $n$.}
		\label{Fig C_4}
\end{figure}

Having collected the proofs for (P1)--(P4), we can now give the proof of (P5) which proves Theorem~\ref{Thm error bound}.

\begin{proof}[Proof of Theorem~\ref{Thm error bound}]
We apply similar steps as in the proof of Proposition~\ref{proposition}, but we switch the roles of $P_f$ and $P_{f,h}$. By Theorem~\ref{Thm. LSF distance} and applying Theorem~\ref{Thm bound of Gaussian measure} with $\widetilde{C}=C_3(h)$, we know that the absolute error of the probability of failure is bounded by 
\begin{align}
\vert P_f -P_{f,h}\vert &\le C_4(n)\mathbb{P}[U_1\in ]-b_h-C_3(h)h^s,-b_h + C_3(h)h^s]].\label{proof part 1}
\end{align} 
Defining the LSF $\widetilde{G}_h(U_1):= U_1 +b_h$ and assuming that $\widetilde{G}$ satisfies~\eqref{ass uniform bound} in Assumption~\ref{Ass LSF error} with $\widetilde{C}_{\mathrm{FE}}= C_3(h)$ yields 
\begin{align}
\mathbb{P}[\widetilde{G}_h(U_1) \in ]-C_3(h)h^s,C_3(h)h^s]] \le C_2(b_h,1, h^s,C_3(h)) h^s P_{f,h}^{\mathrm{FORM}}.\label{proof part 2}
\end{align}
Finally, combining~\eqref{proof part 1} and~\eqref{proof part 2} we conclude the proof of Theorem~\ref{Thm error bound} with
\begin{align*}
\vert P_f -P_{f,h}\vert \le C_2(b_h,1, h^s,C_3(h))\cdot C_4(n) h^s P_{f,h}^{\mathrm{FORM}}=:\widehat{C} h^s P_{f,h}^{\mathrm{FORM}}.
\end{align*}
\end{proof}

\begin{remark}
We note that the assumptions on the regularity of the diffusion coefficient, as given in Assumption~\ref{regularity of a} and~\ref{Ass gaussian}, are only relevant to prove Theorem~\ref{Thm opt control a priori} and~\ref{Thm. LSF distance}, respectively. If the approximation error bound of the LSF behaves in another manner as in Assumption~\ref{Ass LSF error} and if it is possible to show that the distance between $\partial A$ and $\partial A_h$ behaves in the same manner, then the error bounds in Proposition~\ref{proposition} and Theorem~\ref{Thm error bound} hold for this error bound.
\end{remark}

\section{Numerical Experiments}\label{Sec. experiments}
We now illustrate our results in several numerical experiments. We start with a one-dimensional parameter space example where the LSF involves an ODE, not a PDE. Then, we consider rare events that depend on elliptic PDEs with stochastic spaces of different dimensions. In all experiments, the approximation error of $G$ and $G_h$ is not uniformly bounded as required in Assumption~\ref{Ass LSF error}. We consider these settings to test the conjecture in Remark~\ref{Remark error bound}. Indeed, our provided error bounds in Proposition~\ref{proposition} and Theorem~\ref{Thm error bound} are also observed in these non-uniform cases. In the first experiment, we observe that the distances of the failure surfaces behave as the discretization error. Following Remark~\ref{Remark distance}, this behaviour enables to consider only (P4) and (P5) and we expect that the provided error bounds hold in this setting.

\subsection{ODE, 1-dimensional parameter space}\label{Sec 5.1}
In the following example, which is also considered in \cite{Ullmann15}, the LSF depends on the solution of an ODE with a one-dimensional Gaussian random parameter. Hence, this example does not actually depend on an elliptic PDE. We study it, since we can compute all quantities of interest analytically.
\\Let $y:[0,1]\times\Omega\rightarrow \mathbb{R}$ such that for $\mathbb{P}$-a.e. $\omega\in\Omega$, we have
\begin{align*}
\frac{\partial y(t,\omega)}{\partial t} = -U(\omega)y(t,\omega), t\in(0,1), \text{ with initial condition } y(0, \omega)=1,
\end{align*}
where $U\sim \mathrm{N}(0,1)$ is a stochastic parameter. The exact solution to this ODE is given by $y(t,\omega) = \exp(-U(\omega)t)$ -- a $\log$-normal stochastic process. Failure is defined as the event that the solution $y(\cdot, \omega)$ is larger than $y_{\mathrm{max}} = 40$ at $t=1$, which can be written in terms of the LSF $G(u) = y_{\mathrm{max}} - \exp(-u)\le 0$. Hence, failure occurs if $u\le-\log(y_{\mathrm{max}})$ and the exact probability of failure is equal to $P_f=\Phi(-\log(y_{\mathrm{max}}))\approx 1.13\cdot 10^{-4}$. The MLFP is given by $u^{\mathrm{MLFP}} = -\log(y_{\mathrm{max}})$.
\\Using the explicit Euler scheme to solve the ODE numerically, we derive the approximate solution $y_h(t=1, \omega) = (1-U(\omega)h)^{1/h}$, where $h>0$ is the time step size. The explicit Euler scheme is convergent of order one, see \cite[Section 6.3]{LeVeque07}, i.e., $\vert y(t,\omega) - y_h(t,\omega)\vert = \mathcal{O}(h)$ for a fixed $\omega\in\Omega$. The approximate LSF is $G_h(u) = y_{\mathrm{max}}-(1-uh)^{1/h}$. Hence, failure occurs if $u\le (1-y_{\mathrm{max}}^h)/h$ and the approximate probability of failure is equal to $P_{f,h} = \Phi((1-y_{\mathrm{max}}^h)/h)$. The approximate MLFP is given by $u_h^{\mathrm{MLFP}} = (1-y_{\mathrm{max}}^h)/h$.
\\Since the space of the stochastic parameter space is one-dimensional and the exact and approximate failure domains are half-rays, it holds $P_f = P_f^{\mathrm{FORM}}$ and $P_{f,h} = P_{f,h}^{\mathrm{FORM}}$. Thus, the error bound of Proposition~\ref{proposition} and Theorem~\ref{Thm error bound} yield a bound for the relative error $\vert P_f-P_{f,h}\vert /P_f$. As mentioned, the setting of this example is different to the setting of Proposition~\ref{proposition} and Theorem~\ref{Thm error bound}. Moreover, the approximation error of the LSF is not uniformly bounded. Figure~\ref{1D-1D example} shows that the distance of the failure domains scales as $\mathcal{O}(h)$. Thus, we expect that the relative error of the probability of failure approximations has order $\mathcal{O}(h)$ of convergence for $h>0$ sufficiently small.
\\As a second time stepping method, we consider the Crank--Nicolson scheme as given in \cite[Chapter 9]{LeVeque07}. Applying the discretization rule, we get the approximate solution 
\begin{align*}
\widetilde{y}_h(t=1, \omega) = \left(\frac{1-hU(\omega)/2}{1+hU(\omega)/2}\right)^{{1}/{h}}.
\end{align*}
Hence, the approximate probability of failure is given by $\widetilde{P}_{f,h} = \Phi(2h^{-1}(1-y_{\mathrm{max}}^h)/(1+y_{\mathrm{max}}^h))$. Since the Crank--Nicolson scheme is convergent of order $2$, we expect that the relative error of the probability of failure approximations has order $\mathcal{O}(h^2)$ of convergence for $h>0$ sufficiently small.
\\Figure~\ref{1D-1D example} shows the approximate probability of failure by the explicit Euler and the Crank--Nicolson scheme for the step sizes $h_{\ell}=1/2^{\ell}$, for $\ell=0,\dots,9$. We observe that the approximations computed with both these methods approach the exact probability of failure as $h$ decreases. Moreover, we observe that the distance between the exact and approximate MLFPs converges in the same order as the PDE discretization error. Hence, the statement of Theorem~\ref{Thm. LSF distance} is also valid and we conclude that Proposition~\ref{proposition} and Theorem~\ref{Thm error bound} are also applicable for this setting. On the right plot, we observe that the relative error of the explicit Euler approximations has order $\mathcal{O}(h)$ of convergence while the relative error of the Crank--Nicolson approximations has order $\mathcal{O}(h^2)$ of convergence. These are exactly the bounds which we get from theoretical discussions. For large $h$, we observe a plateau behaviour for the explicit Euler scheme. This is due to the fact that the Euler approximation $P_{f,h}$ is much smaller than $P_f$ for large $h$. Hence, the relative error is nearly equal to one for large time step sizes $h$, until the convergence sets in.

\begin{figure}[htbp]
\centering
	\includegraphics[trim=0cm 0cm 0cm 0cm,scale=0.23]{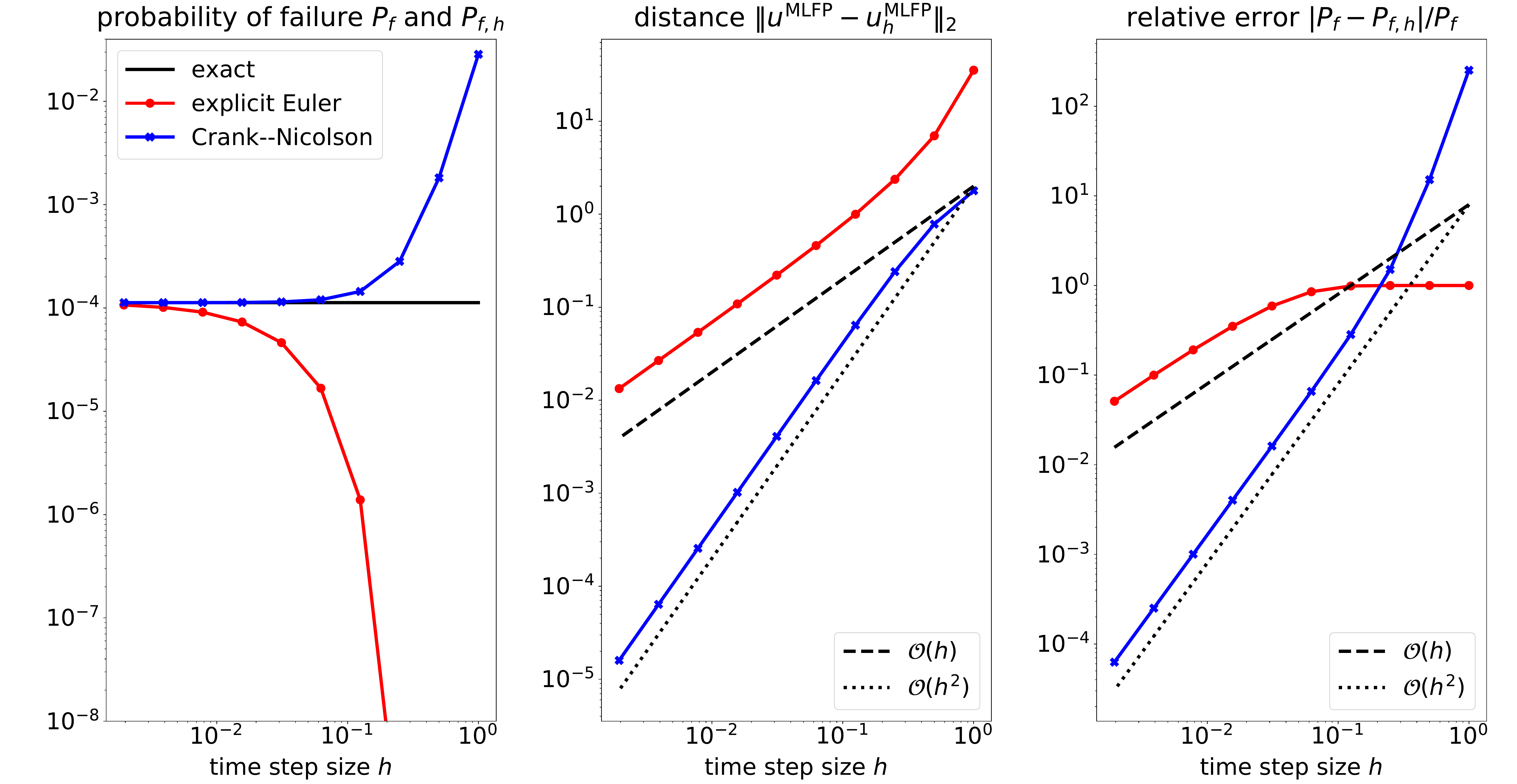}
		\caption{On the left: approximate probability of failure by the explicit Euler and Crank--Nicolson scheme for varying the time step size $h$. The black line shows the exact probability of failure. In the middle: distance between the exact and approximate MLFPs. The black lines show the order of convergence. On the right: relative error of the approximations with respect to the exact probability of failure.}
		\label{1D-1D example}
\end{figure}

\subsection{2-Dimensional parameter space}
The following example considers an LSF which depends on a two-dimensional stochastic parameter and is also considered in \cite{Ernst15,Garbuno20}. In this case, the FORM estimate is not equal to the exact probability of failure. However, we can still derive analytical expressions for the exact and approximate LSF as well as for the exact and approximate limit-state surfaces. On the domain $D=(0,1)$, we seek a solution $y:\overline{D}\times\Omega\rightarrow\mathbb{R}$ that solves the following elliptic boundary value problem 
\begin{align}
- \frac{\partial}{\partial x}\left(\exp\left(\frac{U_1(\omega)}{3}-3\right) \frac{\partial}{\partial x} y(x,\omega)\right) = 1-x, \quad\text{ for } x\in(0,1),\label{BVP}
\\\text{such that } y(0,\omega) = 0 \text{ and } y(1, \omega) = U_2(\omega)\label{BVP boundary}
\end{align}
for $\mathbb{P}$-a.e. $\omega\in\Omega$. The random variables $U_1$ and $U_2$ are independent and standard normally distributed. The exact solution of this problem is
\begin{align*}
y(x,\omega) = U_2(\omega) x+\exp\left(-U_1(\omega)/3+3\right)(x^3/6-x^2/2+x/3).
\end{align*}
Failure is defined as the event that the solution $y(\cdot, \omega)$ is smaller than $y_{\mathrm{max}}= -1/3$ at $\widehat{x}=1/3$. Hence, we express the LSF as $G(U_1(\omega), U_2(\omega)) = y(1/3, \omega) - y_{\mathrm{max}}$. 
\\Applying linear FEs with mesh size parameter $h>0$, we compute the approximate solution to~\eqref{BVP} and~\eqref{BVP boundary} which we denote by $y_h:\overline{D}\times\Omega\rightarrow\mathbb{R}$. Accordingly, the approximate LSF is given by $G_h(U_1(\omega), U_2(\omega)) = y_h(1/3, \omega) - y_{\mathrm{max}}$.
\begin{figure}[htbp]
\centering
	\includegraphics[trim=0cm 0cm 0cm 0cm,scale=0.23]{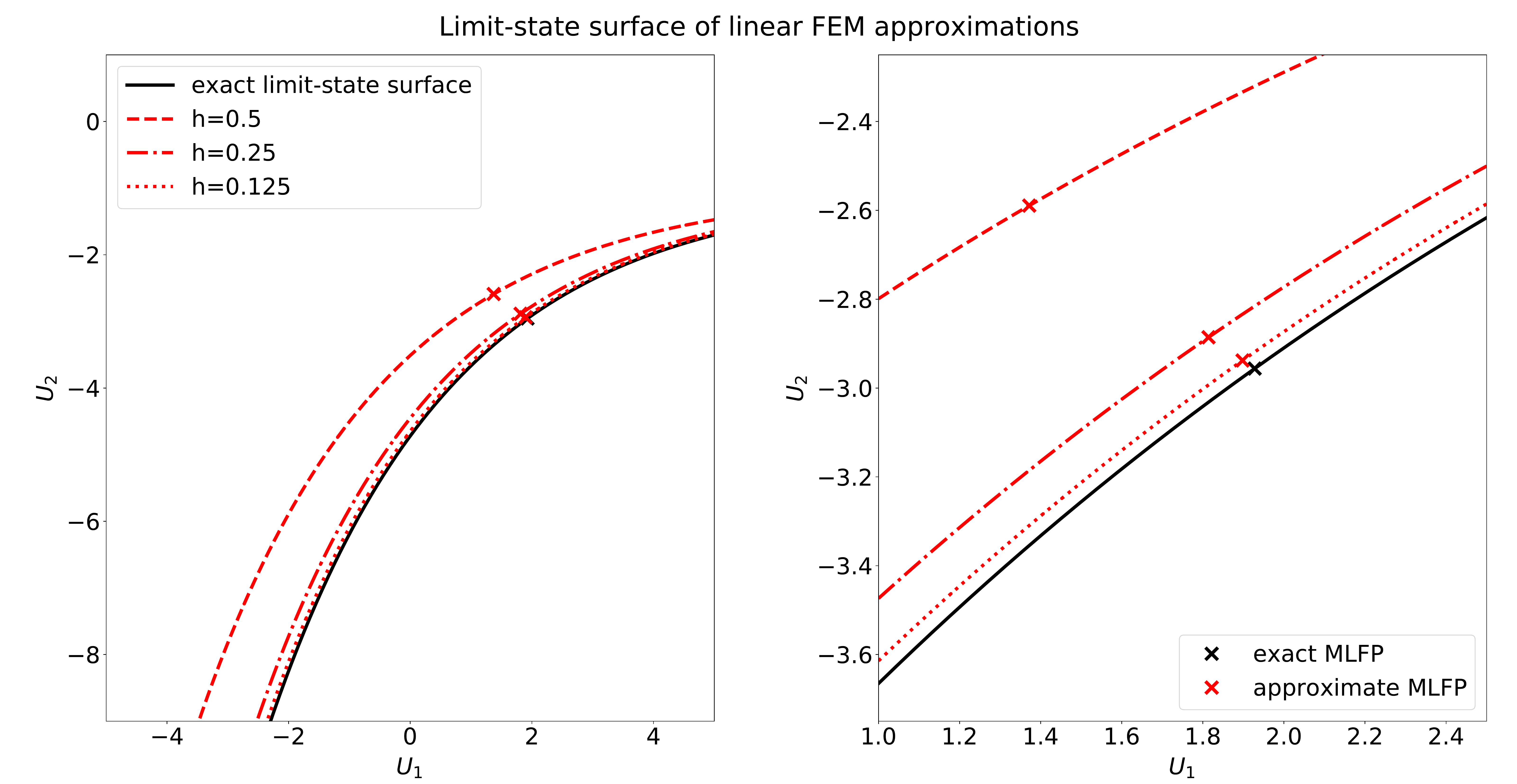}
		\caption{On the left: the black line shows the exact limit-state surface and the black cross denotes the exact MLFP. The red lines show the approximate limit-state surfaces given by the linear FE discretization for different choices of the mesh size $h$. The red crosses denote the approximate MLFPs. On the right: zoom-in of the left plot near the MLFPs.}
		\label{1D-2D example surface}
\end{figure} 
\\Figure~\ref{1D-2D example surface} shows the exact limit-state surface and the limit-state surfaces given by the linear FE approximations. We observe that the exact as well as the approximate failure domains are convex. Indeed, the exact limit-state surface can be expressed as a function in terms of the first coordinate $u_1$ by
\begin{align*}
u_2(u_1) = 1/\widehat{x}\left(-1/3 - \left(\widehat{x}^3/6 - \widehat{x}^2/2+\widehat{x}/3\right)\exp(-u_1/3 + 3)\right).
\end{align*}
Thus, it holds $\partial A = \{(u_1, u_2(u_1)):u_1\in\mathbb{R}\}$. Since $-u_2''(u_1)>0$, we conclude that $-u_2(u_1)$ is a convex function. Since the failure domain $A$ has the same geometric properties as the \emph{epigraph} of $-u_2(u_1)$, we conclude that $A$ is convex. In a similar way, we can prove that $A_h$ is convex. For more details on convex analysis we refer to \cite{Borwein06}.
\\Figure~\ref{1D-2D example surface} also shows that the distances between the exact and approximate surfaces decrease for decreasing mesh size $h$. Since the limit-state surface is not a straight line, the FORM estimates of the probability of failure are not equal to the true ones, i.e., $P_f \neq P_f^{\mathrm{FORM}}$ and $P_{f,h}\neq P_{f,h}^{\mathrm{FORM}}$. The quantities $P_f$ and $P_{f,h}$ are calculated by integrating numerically the standard normal PDF within the failure domain. We obtain the values $P_f\approx 1.71\cdot 10^{-4}$ and $P_{f}^{\mathrm{FORM}}\approx 2.08\cdot 10^{-4}$.
\\Following the theoretical discussions, we expect that
\begin{align*}
\vert P_f - P_{f,h}\vert \le \widehat{C} h^s P_{f,h}^{\mathrm{FORM}},
\end{align*} 
for $h>0$ sufficiently small and $\widehat{C}$ given in the proof of Theorem~\ref{Thm error bound}. The order of convergence $s$ is equal to the order of convergence of the FE discretization. The point evaluation of a linear FE approximation introduces an error of order two, since by \cite[Theorem 1.1]{Douglas75} it holds that the $L^{\infty}$-error is bounded by
\begin{align*}
\Vert y(\cdot, \omega) - y_h(\cdot, \omega)\Vert_{L^{\infty}} \le C\Vert y(\cdot,\omega)\Vert_{W^{2,\infty}} h^2,
\end{align*}
for a fixed $\omega\in\Omega$. Hence, we expect that the error bounds in Proposition~\ref{proposition} and Theorem~\ref{Thm error bound} hold for $s=2$.
\\As another discretization, we apply FEs with quadratic basis functions. In this case, the $L^{\infty}$-error of the exact and FE solution converges with order $\mathcal{O}(h^3)$; see \cite[Theorem 1.1]{Douglas75}. Hence, we expect that the error bounds in Proposition~\ref{proposition} and Theorem~\ref{Thm error bound} hold for $s=3$.
\\Figure~\ref{1D-2D example} shows the error of the probability of failure by linear and quadratic FEs for the mesh sizes $h_{\ell}=1/2^{\ell}$, for $\ell=1,\dots,9$. We observe that the approximate probability of failure $P_{f,h}$ converges to the exact probability of failure $P_f$ for both discretizations. Similarly, $P_{f,h}^{\mathrm{FORM}}$ converges to $P_f^{\mathrm{FORM}}$. The true relative error $\vert P_f -P_{f,h}\vert/P_f$ as well as the relative error of the FORM estimates $\vert P_f^{\mathrm{FORM}}-P_{f,h}^{\mathrm{FORM}}\vert/P_f^{\mathrm{FORM}}$ behaves as the discretization error of the FEM approximations. Moreover, the error bound in Theorem~\ref{Thm error bound} behaves as the discretization error. 

\begin{figure}[htbp]
\centering
	\includegraphics[trim=0cm 0cm 0cm 0cm,scale=0.23]{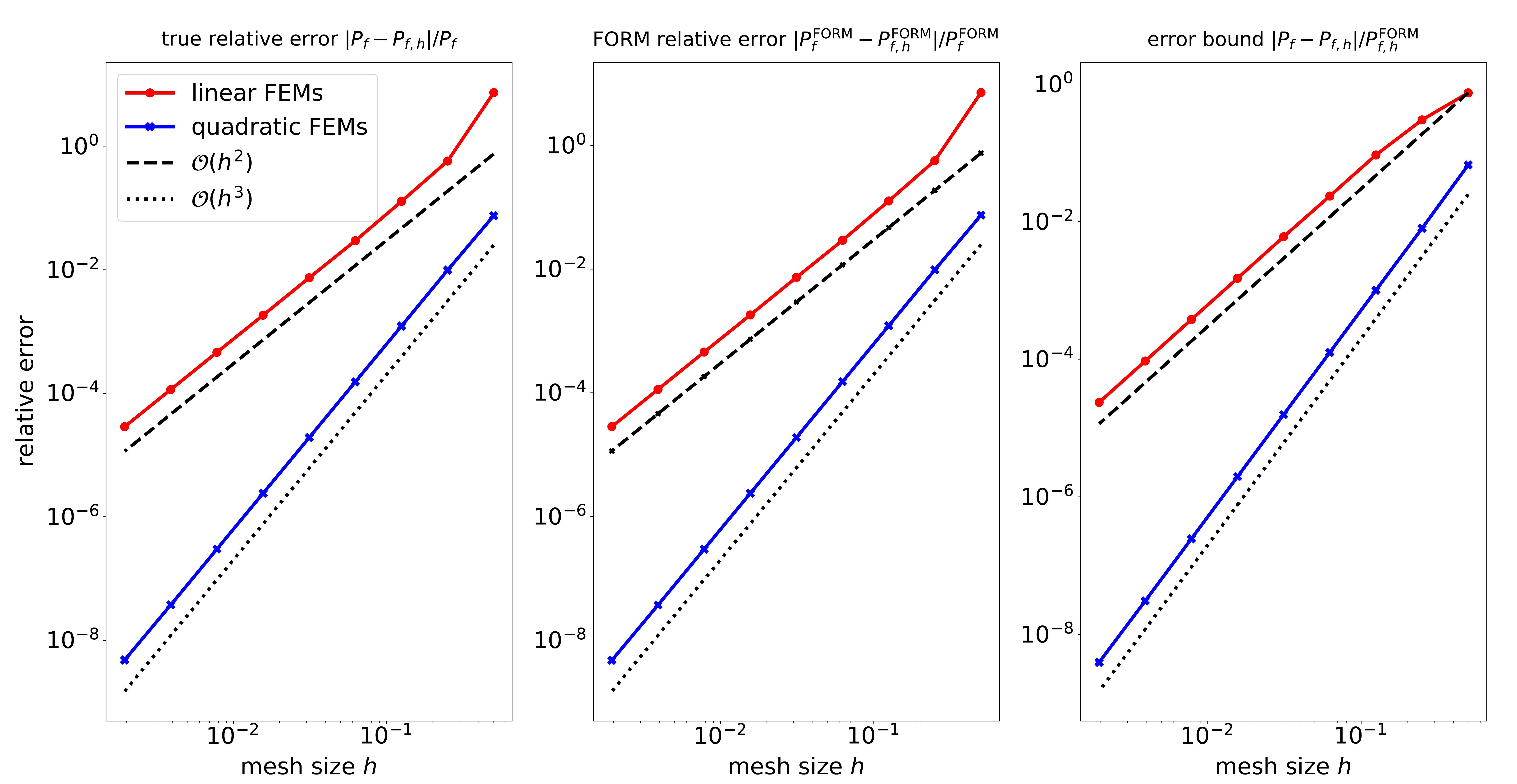}
		\caption{On the left: relative error of the approximate probability of failure $P_{f,h}$ with respect to the exact probability of failure $P_f$. In the middle: relative error of the approximate FORM estimate $P_{f,h}^{\mathrm{FORM}}$ with respect to the exact FORM estimate $P_f^{\mathrm{FORM}}$. On the right: bound of the error $\vert P_f - P_{f,h}\vert$. The black lines show the order of convergence.}
		\label{1D-2D example}
\end{figure}

\subsection{High-dimensional parameter space}\label{Section 5.3}
In the following, we consider Example~\ref{example 1} of Section~\ref{Sec Setting} for which it is not possible to calculate analytic expressions of the exact and approximate PDE solutions $y$ and $y_h$, respectively. Moreover, the limit-state surfaces $A$ and $A_h$ cannot be expressed explicitly. Therefore, we estimate $P_f$ and $P_{f,h}$ with Sequential Importance Sampling (SIS) \cite{Papaioannou16}. 
\\On the domain $D=(0,1)$, we seek a solution $y:\overline{D}\times\Omega\rightarrow\mathbb{R}$ which solves
\begin{align}
-\frac{\partial}{\partial x}\left(a(x, \omega)\frac{\partial}{\partial x} y(x,\omega)\right) &= 0, \quad\text{ for } x\in (0,1),\label{ex 1D diffusion}
\\ \text{such that} \hspace{0.2cm} y(0,\omega) &= 1 \text{ and } y(1,\omega) = 0,\notag
\end{align}
for $\mathbb{P}$-a.e. $\omega\in\Omega$. The random field $a(x,\omega) = \exp(Z(x,\omega))$ is a log-normal random field and the underlying Gaussian field $Z(x,\omega)$ has constant mean $\mu_Z = 0.1$ and variance $\sigma_Z^2=0.04$. The covariance function of $Z$ is $c(x_1,x_2)=\sigma_z^2\exp\left(-\Vert x_1 - x_2\Vert_1/\lambda\right)$, with correlation length $\lambda=0.3$. The random field $Z$ is approximated via its truncated KLE with $n=10$ leading terms, which captures around $93\%$ of the variability of the random field.
Failure is defined as the event that the flow rate $q(\cdot, \omega)$, given in~\eqref{flow rate}, is larger than $q_{\mathrm{max}}= 1.7$ at $\widehat{x}=1$. Hence, we express the LSF as $G(U(\omega)) = q_{\mathrm{max}}- q(1,\omega)$. 
\\Linear FEs are applied with mesh size parameter $h>0$ to obtain the approximate solution $y_h:D\times\Omega\rightarrow\mathbb{R}$ of~\eqref{ex 1D diffusion}. Accordingly, the approximate LSF is given by $G_h(U(\omega)) = q_{\mathrm{max}} - q_h(1,\omega)$. As discussed in Example~\ref{example 1}, linear FEs yield a PDE discretization error of order one. Since the approximation error of the LSF is not uniformly bounded, Proposition~\ref{proposition} and Theorem~\ref{Thm error bound} are not directly applicable. However, as noted in Remark~\ref{Remark error bound}, we expect that our error bounds also hold for $s=1$.
\\The references $P_f$ and $P_{f,h}$ are obtained by averaging over $100$ SIS simulations with $10^4$ samples, target coefficient of variation equal to $0.25$ and using \emph{Markov Chain Monte Carlo} (MCMC) with sampling from the von Mises--Fisher--Nakagami distribution. No burn-in is applied within the MCMC sampling and $10\%$ of the samples are chosen as seeds of the simulated Markov chains via multinomial resampling. Details are given in \cite{Wagner20}. We note that the coefficient of variation of the $100$ probability of failure estimates is $10^{-2}$. Hence, we expect that the sampling bias is negligible. The reference probability of failure is estimated as $P_f \approx 3.38\cdot 10^{-4}$ on a mesh with discretization size $h=2^{-12}$. Similar, the reference FORM estimate $P_f^{\mathrm{FORM}}\approx 4.66\cdot 10^{-4}$ is obtained by FORM with mesh size $h=2^{-12}$. The reference $P_{f,h}$ and $P_{f,h}^{\mathrm{FORM}}$ are obtained on a sequence of mesh sizes $h_{\ell}=1/2^{\ell}$ for $\ell=1,\dots,11$. 
\\The upper left plot of Figure~\ref{1D diffusion} shows the reference probability of failure $P_f$, approximations $P_{f,h}$ and FORM estimates $P_{f,h}^{\mathrm{FORM}}$. We observe that $P_{f,h}^{\mathrm{FORM}}$ is always larger than $P_{f,h}$ for a fixed mesh size $h$. This is a necessary condition for convex failure domains. However, we cannot show that the failure domains are indeed convex and unbounded. The upper right plot shows that the relative error $\vert P_f-P_{f,h}\vert/P_f$ behaves as the discretization error of the LSF. The same holds true for the relative error with respect to the FORM estimates, which is illustrated in the lower left plot. We expected this behaviour by Proposition~\ref{proposition}. Moreover, the lower right plot shows the convergence of our error bound in Theorem~\ref{Thm error bound}. We also observe that the error bound gives an order one approximation which we have expected by Theorem~\ref{Thm error bound}.

\begin{figure}[htbp]
\centering
	\includegraphics[trim=0cm 0cm 0cm 0cm,scale=0.23]{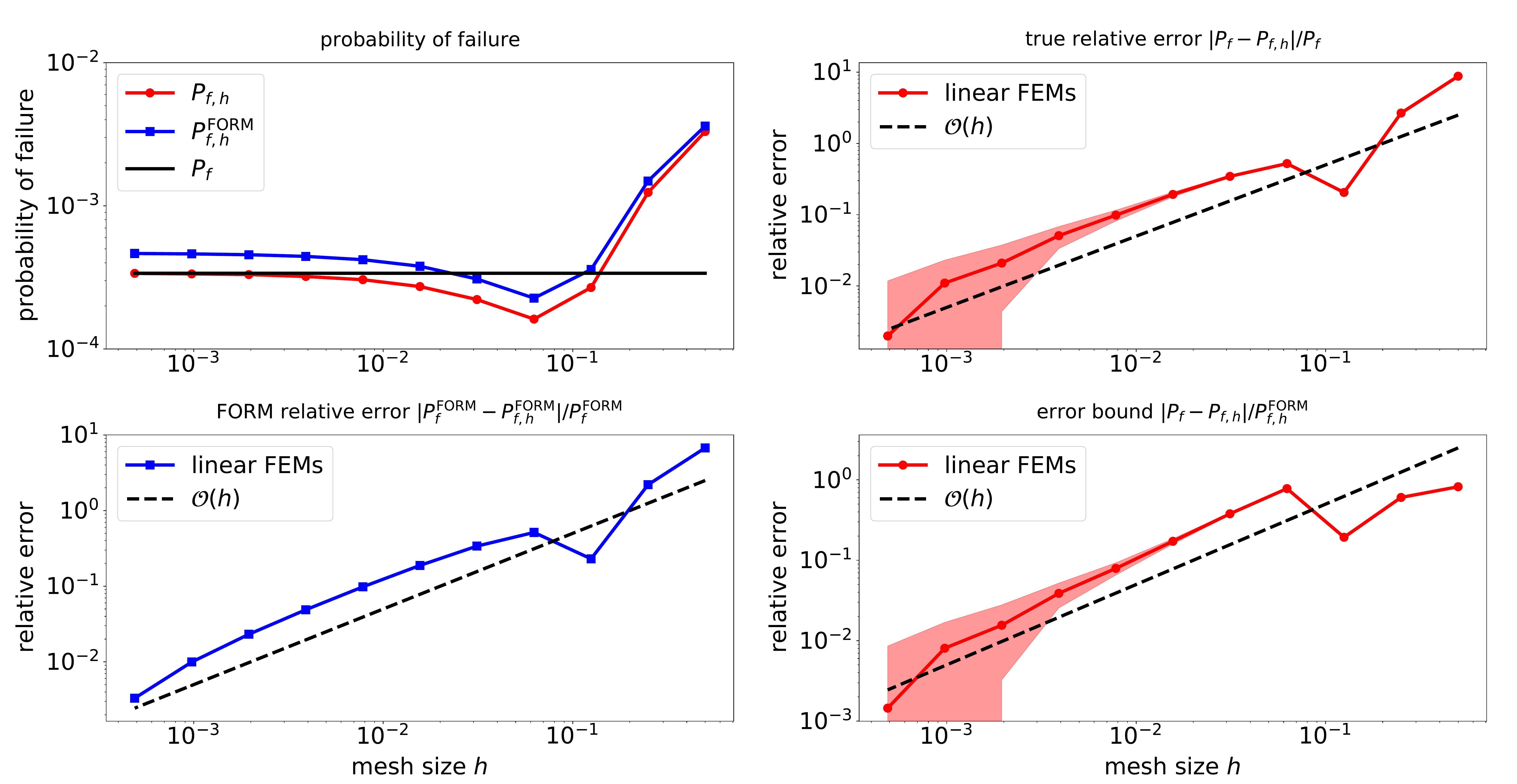}
		\caption{Upper left: reference probability of failure $P_f$, approximations $P_{f,h}$ and FORM estimates $P_{f,h}^{\mathrm{FORM}}$. Upper right: relative error of the approximations $P_{f,h}$ with respect to the reference probability $P_f$. Lower left: relative error of the approximate FORM estimate $P_{f,h}^{\mathrm{FORM}}$ with respect to the reference $P_f^{\mathrm{FORM}}$. Lower right: derived bound of the error $\vert P_f - P_{f,h}\vert$. The dashed black lines show the order of convergence. The red areas show the standard deviations of the estimates.}
		\label{1D diffusion}
\end{figure}

\subsubsection{50-dimensional parameter space}
We consider the problem setting of~\eqref{ex 1D diffusion} with correlation length $\lambda=0.1$. A smaller correlation length requires a larger number of leading KLE terms to acquire a similar resolution of the random field. Therefore, we consider $n=50$ leading KLE terms, which captures around $96\%$ of the variability of the random field. We adjust the threshold $q_{\mathrm{max}}=1.5$ to achieve a similar order of the probability of failure. As in the previous example, the references for the probability of failure are obtained by SIS and the settings as described above. The reference probability of failure is estimated as $P_f = 7.18\cdot 10^{-5}$ on a mesh with discretization size $h=2^{-12}$. The reference FORM estimate $P_f^{\mathrm{FORM}}=1.52\cdot 10^{-4}$ is obtained on the same discretization level. The reference $P_{f,h}$ and $P_{f,h}^{\mathrm{FORM}}$ are obtained on a sequence of mesh sizes $h_{\ell}=1/2^{\ell}$ for $\ell=1,\dots,11$.
\\The upper left plot of Figure~\ref{1D diffusion high dim} shows that $P_{f,h}^{\mathrm{FORM}}$ is always larger than $P_{f,h}$ for a fixed mesh size $h$. The upper right plot shows that the relative error $\vert P_f-P_{f,h}\vert/P_f$ has order $\mathcal{O}(h)$ of convergence for small discretization sizes $h$. For large $h$, we observe a plateau behaviour and then a fast decay until it converges with the expected order. 
\\The relative error with respect to the FORM estimates, which is illustrated in the lower left plot, has order $\mathcal{O}(h)$ of convergence and, hence, is the same as the convergence property of the LSF. Moreover, the lower right plot shows the convergence of our error bound. We also observe that the error bound gives an order one approximation for small $h$. This is exactly the order of convergence we expect from Proposition~\ref{proposition} and Theorem~\ref{Thm error bound}.

\begin{figure}[htbp]
\centering
	\includegraphics[trim=0cm 0cm 0cm 0cm,scale=0.23]{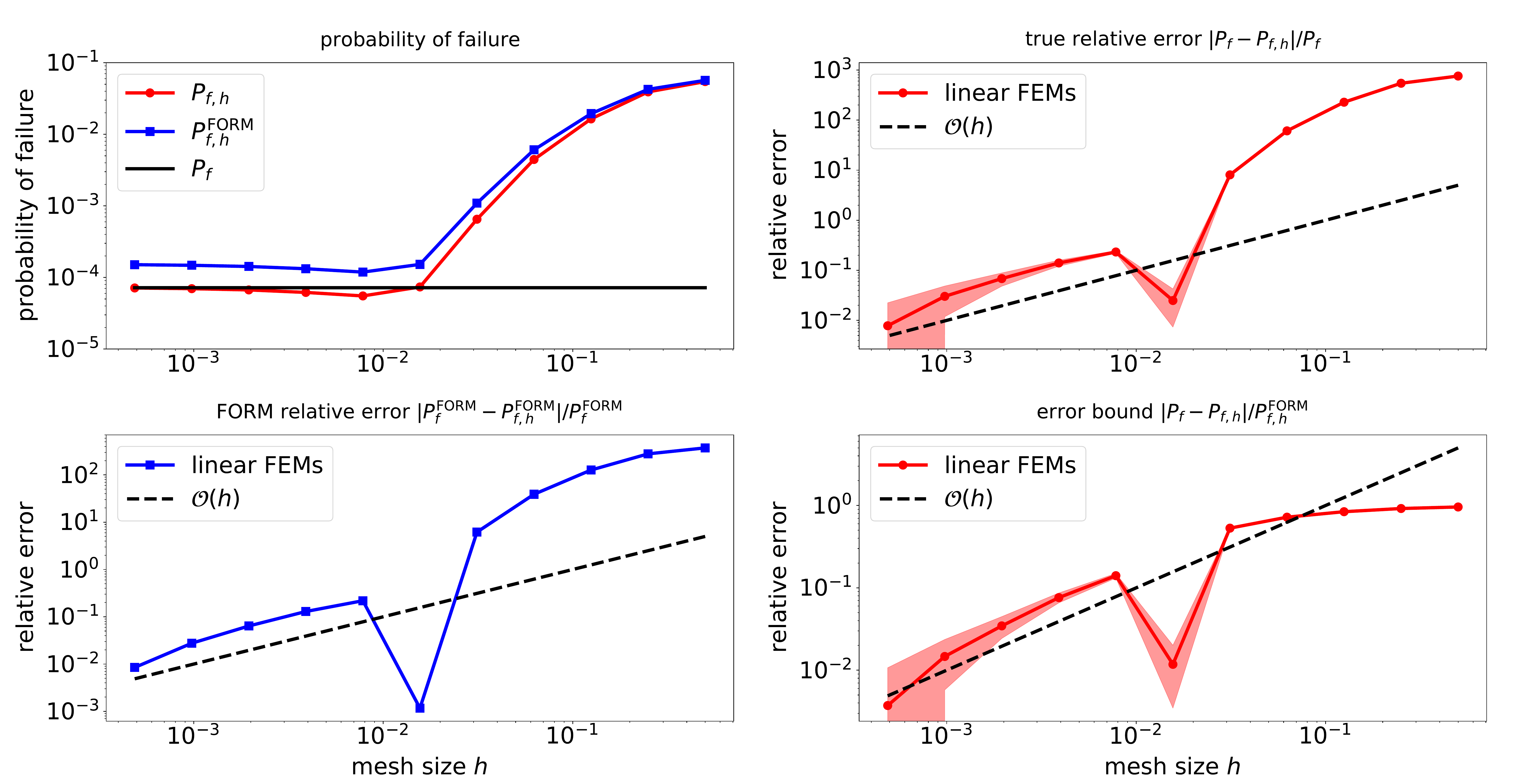}
		\caption{Upper left: reference probability of failure $P_f$, approximations $P_{f,h}$ and FORM estimates $P_{f,h}^{\mathrm{FORM}}$. Upper right: relative error of the approximations $P_{f,h}$ with respect to the reference probability $P_f$. Lower left: relative error of the approximate FORM estimate $P_{f,h}^{\mathrm{FORM}}$ with respect to the reference $P_f^{\mathrm{FORM}}$. Lower right: derived bound of the error $\vert P_f - P_{f,h}\vert$. The dashed black lines show the order of convergence. The red areas show the standard deviations of the estimates.}
		\label{1D diffusion high dim}
\end{figure}

\section{Conclusion and Outlook}\label{Section Conclusion}
In this manuscript, we have considered the approximation error of the probability of failure, which is induced through the approximation error of the LSF. We assume that the LSF depends on the evaluation of an elliptic PDE with stochastic diffusion parameter and Dirichlet boundary condition. We have shown in Theorem~\ref{Thm error bound} under certain assumptions, that the approximation error of the probability of failure behaves as the PDE discretization error multiplied by the FORM estimate of the probability of failure. Moreover, we have shown in Proposition~\ref{proposition} that the relative error of the FORM estimates behaves as the PDE discretization error. If the LSF is affine linear with respect to the stochastic parameter, the derived error bound gives an upper bound for the relative approximation error of the probability of failure. 
\\Our provided error bounds are only applicable for uniformly elliptic and bounded diffusion coefficients. We outline an idea to treat pathwise elliptic and bounded diffusion coefficients. However, we have not provided a complete proof. In several numerical experiments, we observe that our provided error bounds also hold true for pathwise elliptic and bounded diffusion coefficients. In these experiments, we have shown that the approximation error of the probability of failure indeed behaves as the derived error bound given in Theorem~\ref{Thm error bound}. The same holds true for the bound of the relative error of the FORM estimates given in Proposition~\ref{proposition}.
\\The manuscript can be used as a starting point to derive an error bound, which is applicable for a broader range of LSFs. The derivation of an error bound for the relative error, which does not consist of the FORM estimate, is still of high interest.

\section*{Acknowledgments}
We would like to acknowledge the insightful discussion with Daniel Walter about a-priori error estimates for optimal control.

\bibliographystyle{plain}     % mathematics and physical sciences
\bibliography{literatur}   % name your BibTeX data base

\end{document}